\documentclass[12pt]{amsart}
\usepackage{amsmath}
\usepackage[backref=page]{hyperref}
\usepackage{tikz,float,caption}
\usetikzlibrary{arrows.meta}
\usepackage[margin=1.25 in,top=1.2in, bottom=1.4 in]{geometry}

\usepackage[bb=ams,scr=euler]{mathalpha}

\usepackage{setspace}
\usepackage{enumitem}
\setenumerate{
  wide=10pt,
  leftmargin=0pt,
  labelwidth=*,
  label=(\roman*),
  topsep=0pt,
  listparindent=0pt,
  parsep=4pt}

\usepackage{thmtools}
\declaretheoremstyle[headfont=\normalsize\normalfont\bfseries,notefont=\mdseries,
notebraces={(}{)},bodyfont=\normalfont,postheadspace=0.5em]{basicstyle}
\declaretheoremstyle[headfont=\normalsize\normalfont\bfseries,notefont=\mdseries,
notebraces={(}{)},bodyfont=\normalfont\itshape,postheadspace=0.5em]{italstyle}

\declaretheorem[name=Definition,style=basicstyle]{defn}

\declaretheorem[style=italstyle,name=Theorem]{theorem}
\declaretheorem[style=italstyle,name=Question]{question}
\declaretheorem[style=italstyle,name=Corollary,numbered=no]{cor}

\declaretheorem[style=italstyle,name=Proposition,sibling=theorem]{prop}
\declaretheorem[style=italstyle,name=Conjecture,sibling=question]{conj}
\declaretheorem[style=italstyle,name=Lemma,sibling=theorem]{lemma}
\declaretheorem[style=basicstyle,name=Proof,numbered=no,qed=$\square$]{preproof}
\renewenvironment{proof}{\preproof}{\endpreproof}

\makeatletter
\newcommand{\abs}[1]{\left|#1\right|}
\newcommand{\bd}{\partial}
\newcommand{\C}{\mathbb{C}}
\renewcommand{\d}{\mathrm{d}}
\newcommand{\id}{\mathrm{id}}
\newcommand{\intprod}{\mathbin{{\tikz{\draw(-0.1,0)--(0.1,0)--(0.1,0.2)}\hspace{0.5mm}}}}

\newcommand{\pd}[2]{\frac{\partial #1}{\partial #2}}
\newcommand{\pr}{\mathrm{pr}}
\newcommand{\R}{\mathbb{R}}
\renewcommand{\setminus}{\mathbin{\tikz[baseline={(0,0.033)}]{\draw[line width=0.4, rounded corners=0.5,line cap=round,scale=0.7] (0,0.3)--(0.25,0.1);}}}
\newcommand{\set}[1]{\left\{#1\right\}}
\renewcommand{\subsection}{\@startsection{subsection}{2}%
  \z@{.5\linespacing\@plus.7\linespacing}{-.5em}%
  {\normalfont\itshape}}
\newcommand{\ud}[2]{\frac{\mathrm{d} #1}{\mathrm{d} #2}}
\newcommand{\Z}{\mathbb{Z}}

\makeatother

\title{Remarks on the Oscillation Energy of Legendrian Isotopies}
\author{Dylan Cant}
\date{\today}
\begin{document}
\maketitle

\begin{abstract}
  We construct non-compact contact manifolds containing compact Legendrians which can be displaced from their Reeb flow with arbitrarily small oscillation energy. We use this to show the Shelukhin-Chekanov-Hofer pseudo-metric considered by \cite{zhang_rosen} is degenerate on the isotopy class of the constructed Legendrians. Other aspects related to oscillation energy of Legendrian isotopies are explored.
\end{abstract}

\section{Introduction}
\label{sec:intro}

A \emph{Legendrian} in a contact manifold $(Y,\xi)$ is a submanifold $\Lambda^{n}\subset Y^{2n+1}$ which is tangent to the contact distribution $\xi$. A Legendrian isotopy is a one-parameter family of Legendrian submanifolds. Given a choice of contact form $\alpha$ for $(Y,\xi)$, we will associate quantitative measurements to Legendrian isotopies, modeled on the measurements introduced by \cite{shelukhin_contactomorphism} for contactomorphisms. Similar quantitative measurements have been studied by multiple authors, see for instance \cite{zhang_rosen,usher_local_rigidity,hedicke,oh_legendrian_entanglement,rizell_sullivan_conformal,rizell_sullivan_persistence_1,rizell_sullivan_augmentation,rizell_sullivan_pos_loops,akaho01,her07}. The measurements we consider can be considered as Legendrian versions of the Hofer energy, \cite{hofer_energy}, for exact isotopies of Lagrangians, as in \cite{polterovich_lag_displacement_energy,chekanov_1998,chekanov_2000}.

Suppose $\Lambda_{t}$, $t\in [0,1]$, is an isotopy of compact Legendrians. Pick a family of smooth embeddings $i_{t}:\Lambda\to Y$ so that $i_{t}$ parametrizes $\Lambda_{t}$. Write $i'_{t}$ for the time derivative of $i_{t}$ and define the \emph{contact Hamiltonian} $h_{i,t}:\Lambda\to \R$ by the formula:
\begin{equation}\label{eq:contact_hamiltonian}
  h_{i,t}(x)=\alpha_{i_{t}(x)}(i'_{t}(x)).
\end{equation}
We define the \emph{length} and the \emph{oscillation energy} of the isotopy to be the integrals:
\begin{equation}\label{eq:length_osc_energy}
  \begin{aligned}
    \mathrm{length}_{\alpha}(\Lambda_{t})&=\int_{0}^{1}\max_{x\in \Lambda} \abs{h_{i,t}(x)}\d t,\\
    \mathrm{energy}_{\alpha}^{\mathrm{osc}}(\Lambda_{t})&=\int_{0}^{1}\max_{x\in \Lambda} h_{i,t}(x)-\min_{x\in \Lambda}h_{i,t}(x)\d t. 
  \end{aligned}
\end{equation}
Any reparametrization of the isotopy $j_{t}=i_{t}\circ \sigma_{t}$ has $h_{j,t}=h_{i,t}\circ \sigma_{t}$, and hence these quantities are independent of the choice of parametrization. Let us write $h_{t}$ when the parametrization is clear.

Taking infima over isotopies joining $\Lambda_{0}$ to $\Lambda_{1}$ yields a (pseudo) distance function,
\begin{equation*}
  \mathrm{dist}_{\alpha}(\Lambda_{0},\Lambda_{1}):=\inf\set{\mathrm{length}_{\alpha}(\Lambda_{t}):\Lambda_{t}\text{ isotopy from $\Lambda_{0}$ to $\Lambda_{1}$}}.
\end{equation*}

Standard isotopy extension results imply that for any Legendrian isotopy $i_{t}$, there exists an ambient isotopy $\varphi_{t}$ and a reparametrization $\sigma_{t}$ so that $i_{t}=\varphi_{t}\circ i_{0}\circ \sigma_{t}$. Indeed, if one extends the functions $h_{t}$ by finding $H_{t}$ so that $h_{t}=H_{t}\circ i_{t}$, then the contactomorphism generated by $H_{t}$ will be the desired extension\footnote{One can arrange that $\sigma_{t}=\mathrm{id}$, see \cite[\S2.3]{traynor_helix_links}.} $\varphi_{t}$. Using a tubular neighborhood construction, one can arrange that:
\begin{equation*}
  \max_{y\in Y}H_{t}=\max_{x\in \Lambda}h_{t}\hspace{.3cm}\text{ and }\hspace{.3cm}\min_{y\in Y}H_{t}=\min_{x\in \Lambda} h_{t}.
\end{equation*}
This yields a comparison with the norms $\abs{\varphi}_{\alpha}$ and  $\abs{\varphi}_{\alpha}^{\mathrm{osc}}$ defined in \cite{shelukhin_contactomorphism},
\begin{equation*}
  \mathrm{dist}_{\alpha}(\Lambda_{0},\Lambda_{1})=\inf\set{\abs{\varphi}_{\alpha}:\varphi(\Lambda_{0})=\Lambda_{1}}.
\end{equation*}
The distance function $\mathrm{dist}_{\alpha}$ is called the \emph{Shelukhin-Chekanov-Hofer pseudo-metric} in \cite{zhang_rosen}. One considers it as a pseudo-metric on the isotopy class of a Legendrian. Using the dichotomy established in \cite{zhang_rosen}, the paper \cite{rizell_sullivan_augmentation} proves this pseudo-metric is non-degenerate on all isotopy classes of compact Legendrians, in the class of contact manifolds where the Chekanov-Eliashberg DGA technology\footnote{As in \cite{rizell_sullivan_pos_loops}, this includes all closed contact manifolds $Y$ and any compact codimension $0$ submanifold-with-boundary in $Y$.} is established. The non-degeneracy of the metric was established previously in \cite{usher_local_rigidity} for hypertight Legendrians, and in \cite{hedicke} for the case of \emph{orderable} Legendrian isotopy classes. One of the results in this paper is the construction of a Legendrian isotopy class for which the pseudo-metric is degenerate (Theorem \ref{theorem:nonLDR}). This degeneracy shows that \cite[Conjecture 1.11]{zhang_rosen} (that the metric is always non-degenerate) does not hold in generality; one needs to restrict the statement to a suitable class of tame contact manifolds (with suitably tame contact forms).

A contactomorphisms is called \emph{strict} if it satisfies $\phi^{*}\alpha=\alpha$. For the purposes of this note, a Legendrian isotopy $\Lambda_{t}$ is called \emph{strict} if it is induced by an isotopy of strict contactomorphisms $\phi_{t}$.

\subsection{LDR contact manifolds}
\label{sec:LDR_disjoin}

The oscillation energy is related to the problem of finding Legendrian isotopies $\Lambda_{t}$ so that there are no Reeb chords between $\Lambda_{0}$ and $\Lambda_{1}$ (henceforth we refer to this as \emph{disjoining} the Legendrian). For a certain class of contact manifolds, there is a minimum oscillation energy needed to disjoin a Legendrian. This can be formalized with the following definition:

\begin{defn}\label{defn:LDR}
  A contact manifold $(Y,\alpha)$ is \emph{Legendrian disjoinment rigid} (LDR) if for all isotopies $\Lambda_{t}$:
  \begin{equation}\label{eq:LDR}
    \text{energy}^{\mathrm{osc}}_{\alpha}(\Lambda_{t})<C(\alpha,\Lambda_{0})\implies \text{there is a Reeb chord between $\Lambda_{0}$ and $\Lambda_{1}$,}
  \end{equation}
  where $C(\alpha,\Lambda_{0})$ is some positive constant depending on $\alpha$ and $\Lambda_{0}$. In words, in an LDR contact manifold, isotopies require a minimum amount of oscillation energy to disjoin a given Legendrian. It is important to note that the LDR notion depends on the choice of form, see \S\ref{sec:LDR_comment} for further discussion. 
\end{defn}
In an analogous fashion, $(Y,\alpha)$ is \emph{strict LDR} if it satisfies \eqref{eq:LDR} for all strict contact isotopies.

Let us say that a contact manifold $Y$ is \emph{quasi-fillable} if $SY$ can be symplectically embedded into a tame symplectic manifold $X$. We say that $Y$ is \emph{aspherically quasi-fillable} if we can take $X$ with $\pi_{2}(X,SY)=0$, as in \cite{klaus_chord}. Our first result is:

\begin{theorem}\label{theorem:mohnke}
  Quasi-fillable contact manifolds $Y$ are strict LDR for any choice of contact form $\alpha$. If $Y$ is aspherically quasi-fillable and the contact form has a complete Reeb flow, then we can take $C(\Lambda,\alpha)$ to be the minimal positive action of a Reeb chord of $\Lambda$.
\end{theorem}

Results such as this, guaranteeing the existence of chords subject to a bound on the oscillation energy, can be considered as a contact version of Chekanov's result \cite{chekanov_1998} on the persistence of Lagrangian intersections. See \cite[Theorem 1.2]{rizell_sullivan_conformal}, \cite[Theorem 1.1]{rizell_sullivan_persistence_1}, \cite[Theorem 1.13]{oh_legendrian_entanglement}, \cite{rizell_sullivan_augmentation} for results in this vein.

The persistence of Reeb chords appears earlier in the literature with \cite{EHS95}, \cite{ono_lag_leg}, where the authors prove that, for certain data $(Y,\alpha,\Lambda_{0})$, and arbitrary isotopies $\Lambda_{t}$, there are Reeb chords between $\Lambda_{0}$ and $\Lambda_{1}$. The authors do not use the language of Reeb chords, but rather of intersection points of $\Lambda_{t}$ with a \emph{pre-Lagrangian} $L$ containing $\Lambda_{0}$. In fact, $L$ is foliated by Legendrians, and $\Lambda_{0}$ is a particular leaf. In the cases considered one can pick a contact form whose Reeb flow preserves $L$ and induces a free $S^{1}$-action so that $\Lambda_{0}$ is a section of the action. The authors prove that intersections between $\Lambda_{0}$ and $L$ persist under arbitrary contact isotopies of $\Lambda_{0}$. The proofs are based on two Floer homology theories counting intersection points between $S\Lambda_{0}$ and $L$ in $SY$. In similar settings, the papers \cite{akaho01}, \cite{her07} prove persistence results which incorporate the oscillation energy.

Note that \cite{chekanovQF}'s result on Legendrian isotopies of the zero-section in $1$-jet space can be interpreted as a persistence of Reeb chords.

\subsubsection{On strict LDR versus general LDR and short chords}
\label{sec:strict}

The strictness of the isotopy $\Lambda_{t}$ is a technical requirement in our geometric argument for Theorem~\ref{theorem:mohnke}. We require a lower bound on the length of Reeb chords of $\Lambda_{t}$ depending only on $\Lambda_{0}$. For general isotopies no such lower bound exists, as chords of arbitrarily short length can be born. The appearance of short chords also poses a problem when considering holomorphic curves in symplectizations.

As an example, consider counts of holomorphic strips in $SY$ with moving boundary conditions (i.e., boundary values on the Lagrangians $\set{S\Lambda_{t}:t\in [0,1]}$ in a prescribed way, e.g., $u(s,0)\in S\Lambda_{0}$, $u(s,1)\in S\Lambda_{f(s)}$, for some $f:\R\to [0,1]$). To obtain compactness of such moduli spaces, one needs to handle bubbling (in the sense of \cite{BEHWZ}) along chords of \emph{any} of the intermediate Legendrians $\Lambda_{t}$. Thus, the shortest chords which appear during the isotopy are relevant to the study of such moduli spaces. The work of \cite{oh_legendrian_entanglement} establishes a priori estimates for the $\d\alpha$-energy of such holomorphic curves in terms of the oscillation energy of isotopies. 

In another direction, one can consider \emph{Lagrangian cobordisms} $L$ in $SY$ interpolating between $\Lambda_{0}$ and $\Lambda_{1}$. The papers \cite{sabloff_traynor_length,ekholm_honda_kalman} explain how to use an isotopy $\Lambda_{t}$ to construct such a cobordism. The length of the shortest Reeb chord of $\Lambda_{t}$ determines the ``length'' of the constructed cobordism.\footnote{Recall that, using the contact form $\alpha$ to provide a vertical coordinate $s_{\alpha}:SY\to \R$, the \emph{length} of a Lagrangian $L$ is the length of the minimal interval $I$ so that $L$ is cylindrical (tangent to the Liouville field) in the region where $s_{\alpha}\not\in I$.} The length of the cobordism then appears in a priori estimates for holomorphic curves with boundary on $L$. This is the approach taken in \cite{rizell_sullivan_conformal}.

\subsection{Non-LDR contact manifolds}
\label{sec:non-LDR}

There are contact manifolds which are not strict LDR (and thus are also not LDR). The construction is a familiar one: given two contact manifolds $Y_{0},Y_{1}$ and a choice of form $\alpha_{0}$ for $Y_{0}$ we define the (generalized) \emph{contactization} to be the product $Q=Y_{0}\times SY_{1}$ with the contact form $A=\alpha_{0}-\lambda_{1}$ (often this is considered with $Y_{0}=\R$ or $S^{1}$). 
\begin{theorem}\label{theorem:nonLDR}
  For certain exotic choices of $Y_{1}$, and arbitrary $Y_{0}$, the contactizations $(Q,A)$ contain compact Legendrians $\Lambda$ which can be disjoined by strict contact isotopies with arbitrarily small oscillation energy. Moreover, the Shelukhin-Chekanov-Hofer pseudo-metric for the contact form $A$ is degenerate on the space of Legendrians isotopic to $\Lambda$.
\end{theorem}
The dichotomy proved in \cite{zhang_rosen} immediately implies:
\begin{cor}
With $Q,A,\Lambda$ as in Theorem \ref{theorem:nonLDR}, the Shelukhin-Chekanov-Hofer pseudo-metric is identically zero on the space of Legendrians isotopic to $\Lambda$.
\end{cor}

Our proof of Theorem \ref{theorem:nonLDR} is based on Sikorav's example, as conveyed in \cite[\S4]{chekanov_2000}. Indeed, our ``exotic'' contact manifolds are those $Y_{1}$ for which the symplectization $SY_{1}$ admits a \emph{compact embedded exact Lagrangian}. The results of \cite{chekanov_2000} imply that such symplectizations cannot be embedded into any tame symplectic manifold (and hence we deem these $Y_{1}$ exotic). Sikorav's argument uses an exotic contact $\R^{5}$ whose symplectization admits a compact Lagrangian sphere due to \cite{marie_paule_muller}.

Using the $h$-principle for Lagrangian caps established in \cite{lagrangian_caps}, Murphy explains in \cite{murphy_closed_exact} that there are many compact $Y_{1}$, with $\dim(Y_{1})\ge 5$, whose symplectizations contain compact exact Lagrangians; the requirement is the existence of a certain \emph{plastikstufe} in $Y_{1}$. The papers \cite{BEM,casals_murphy_presas} relate the existence of a plastikstufe to the \emph{overtwistedness} of $Y_{1}$. This suggests the following questions:

\begin{question}
  Is there a three-dimensional compact contact manifold whose symplectization contains a compact exact Lagrangian? Note: Any such Lagrangian is displaceable and hence is necessarily a 2-torus.
\end{question}
\begin{question}
  Are there tight contact manifolds whose symplectizations contain compact exact Lagrangians?
\end{question}

\subsubsection{Dependence of the LDR property on the contact form}
\label{sec:LDR_comment}
The argument proving Theorem \ref{theorem:nonLDR} implies that $\R^{7}\simeq \R\times S\R^{5}$ contains Legendrian spheres which can be disjoined with arbitrarily small oscillation energy, where $\R^{5}$ has the exotic contact structure constructed in \cite{marie_paule_muller}. Then \cite[Corollary 1.4]{BEM} implies that this exotic $\R^{7}$ can be embedded into any $7$-dimensional overtwisted manifold, in particular, it can be embedded into a \emph{closed} contact manifold. The work of \cite{rizell_sullivan_persistence_1,rizell_sullivan_augmentation} implies that closed contact manifolds are LDR.

\subsection{Disjoinment rigidity for Legendrians admitting generating functions}
\label{sec:general_rigidity}

The next result is that certain Legendrians in $1$-jet spaces are disjoinment rigid. The relevant class of Legendrians are those which admit \emph{linear-at-infinity generating functions}, as in \cite{chekanov_pushkar,jordan_traynor,sabloff_traynor_slice,fuchs_rutherford,sabloff_traynor_obstructions,bourgeois_sabloff_traynor}. A function $F:B\times \R^{N}\to \R$, is called \emph{linear-at-infinity} if $F(x,\eta)=f(x,\eta)+\ell(\eta)$ where $\ell\ne 0$ is a linear function and $f$ is compactly supported. Let us say that $F$ is a \emph{generating function} provided the fiber derivative $\bd_{\eta}F:B\times \R^{N}\to \R^{N}$ has $0$ as a regular value. In this case, $\bd_{\eta}F=0$ defines a smooth fiberwise singular set $\Sigma$, and the map $\Sigma\to J^{1}(B)$ (returning the critical value and horizontal derivative) defines an immersed Legendrian. A Legendrian which can be parametrized by such a map $\Sigma\to J^{1}(B)$ is said to admit a generating function. As detailed in \cite[\S3]{fuchs_rutherford}, for Legendrian knots in $\R^{3}$ the existence of a linear-at-infinity generating function is equivalent to the existence of a graded normal ruling of the front projection. The paper \cite{bourgeois_sabloff_traynor} gives many examples of Legendrians which admit linear-at-infinity generating functions. The main result of this section is:

\begin{theorem}\label{theorem:generating_function}
  Suppose that $\Lambda_{0}\subset J^{1}(B)$ is an embedded Legendrian which admits a linear-at-infinity generating function. Then any isotopy $\Lambda_{t}$ disjoining $\Lambda_{0}$ has oscillation energy at least the minimal positive action of a Reeb chord of $\Lambda_{0}$. 
\end{theorem}

The crux of the argument is to construct a \emph{barcode}\footnote{in the sense of topological data analysis; see \cite{ghrist_persistent} for further details and \cite{polterovich_shelukhin_persistence_1,usher_zhang,kislev_shelukhin,shelukhin_viterbo,shelukhin_zoll} and the references therein for applications to symplectic geometry.} so that the endpoints of the bars are actions of Reeb chords between $\Lambda_{0}$ and $\Lambda_{t}$. The lengths of the bars are Lipshitz continuous as a function of $t$, and the Lipshitz constant depends on the oscillation energy of the isotopy. If there are no chords between $\Lambda_{0}$ and $\Lambda_{1}$, then the barcode at time $t=1$ is empty. In this fashion, we will conclude the minimal amount of oscillation energy needed to disjoin $\Lambda_{0}$ is at least \emph{the length of longest bar in the barcode at time $t=0$}. The argument is completed by showing that there are bars at $t=0$ whose length is at least the minimal positive action of a chord of $\Lambda_{0}$.

Our argument is inspired by the papers \cite{rizell_sullivan_persistence_1} and \cite{rizell_sullivan_augmentation} which establish the LDR property in certain cases by examining how the Chekanov-Eliashberg DGA changes under bifurcations. Via an analysis of the possible bifurcations, they show that one can define a continuously varying barcode and perform a similar argument to the one outlined above. Strictly speaking, they do not construct a barcode in the usual sense, but rather a more general barcode incorporating an action window whose endpoints depend on $t$; see \cite[\S2.1]{rizell_sullivan_persistence_1}. Their persistence results apply in more general settings than ours, as the Chekanov-Eliashberg DGA theory is well-established in $1$-jet spaces by the work of \cite{ekholm_etnyre_sullivan_P}. Hopefully our generating function approach is of independent interest. 

% \subsection{A criterion for non-degeneracy of the pseudo-metric}
% \label{sec:strongly_LDR}

% Our third result is a simple construction which can be used to prove non-degeneracy of the Shelukhin-Chekanov-Hofer pseudo-metric:

% \begin{theorem}\label{theorem:LDR_consequence}
%   Suppose that $(Y,\alpha)$ has a periodic Reeb flow and $Y\times \C$ is LDR for the form $\alpha-\lambda$ where $\lambda=\frac{1}{2}(x\d y-y\d x)$. Then the Shelukhin-Chekanov-Hofer metric for $\alpha$ is non-degenerate for all compact Legendrian isotopy classes in $Y$.
% \end{theorem}

% The proof we present makes crucial use of the dichotomy from \cite{zhang_rosen}, and is heavily inspired by the arguments in \cite{chekanov_2000}. 

\subsection{Oscillation energy for regular homotopies of Legendrians}
\label{sec:immersions}

The notions of length and oscillation energy can be generalized to families of immersions, i.e., \emph{regular homotopies} of Legendrians, as in \cite{laudenbach_regular_homotopies}. Given a family of Legendrian immersions $i_{t}:\Lambda\to (Y,\alpha)$, we define the contact Hamiltonian $h_{t}(x)=\alpha_{i_{t}(x)}(i_{t}'(x))$, length, and oscillation energy exactly as in \eqref{eq:length_osc_energy}. The difference when compared with embeddings is that regular homotopies cannot generally be extended to ambient contact isotopies.

Our first result about Legendrian immersions is a construction of a Legendrian which can be disjoined by a regular homotopy with arbitrarily small oscillation energy. Introduce the following local model: $B(a)\subset \C$ is the disk of area $a$, and $Q(a)$ is the contactization:
\begin{equation*}
  Q(a)=\R/\Z\times B(a)\text{ with contact form }\alpha=\d t-\lambda,
\end{equation*}
where $\lambda=\frac{1}{2}(x\d y-y\d x)$ is the standard Liouville form. 

\begin{theorem}\label{theorem:immersion_disjoining}
  Let $(Y,\xi)$ be a contact three-manifold which contains an embedded copy of $Q(1+\epsilon)$ for some $\epsilon>0$. Let $\alpha$ be a contact form which agrees with the $\d t-\lambda$ form on $Q(1+\epsilon')$ for some $0<\epsilon'<\epsilon$. Then $(Y,\alpha)$ contains an embedded Legendrian $\Lambda_{0}$ so that there are regular homotopies $\Lambda_{t}$ of arbitrarily small oscillation energy, with no Reeb chords between $\Lambda_{0}$ and $\Lambda_{1}$.
\end{theorem}

Note that \cite[Corollary 1.25]{ekp} states that every contact three-manifold contains an embedded copy of $Q(a)$ for \emph{every} $a>0$. Thus our theorem applies to all contact three-manifolds.

\subsubsection{On Legendrian non-squeezing}
\label{sec:conjecture}
The Legendrian used in Theorem \ref{theorem:immersion_disjoining} is a lift of the loop $\bd B(1)$ (note that $\R/\Z\times \bd B(1)$ is foliated by Legendrian lifts, and $\R/\Z$ acts freely on the set of lifts). The argument used in \S\ref{sec:proof_immersion} shows the Legendrian can be squeezed inside $\R/\Z\times B(1)$ if one allows immersions. The following result shows this Legendrian cannot be squeezed into $\R/\Z\times B(1)$ through embeddings, strengthening a non-squeezing result of \cite{ekp} (in the 3-dimensional case).
\begin{prop}\label{prop:rizell_sullivan}
  Let $\Lambda_{1}\subset \R/\Z\times \C$ be a Legendrian lift of $\bd B(1)\subset \C$. Then $\Lambda_{1}$ is not isotopic to a Legendrian contained in $\R/\Z\times B(1)$.
\end{prop}
Originally this was stated as a conjecture, but its proof was communicated to me by \cite{rizell_sullivan_private_communication}, using results from \cite{rizell_sullivan_non_squeezing}. The proof is given in \S\ref{sec:proof_conjecture}.

% If this conjecture is true, it would give an interesting perpective on contact non-squeezing.

% Here is a partial result in the direction of the above conjecture. Observe that: $$P:(z,q,p)\in \R\times \R\times (-1,\infty)\mapsto (z+q,(p+1)^{1/2} \pi^{-1/2}e^{2\pi i q})\in \R/\Z\times \C^{\times},$$
% is a covering map which satisfies $P^{*}(\d t-\lambda)=\d z-p\d q$, and $P$ maps the zero wall $p=0$ onto the torus $\R/\Z\times \bd B(1)$ and the zero section $z=p=0$ onto $\Lambda$. Thus we conclude from Chekanov's result \cite{chekanovQF} (and the homotopy lifting property for covering maps) that $\Lambda$ cannot be displaced from $\R/\Z\times \bd B(1)$ by a contact isotopy supported in the complement of $\R/\Z\times \set{0}$. See \cite[\S2.2.1]{eliashberg_wave_fronts} for related discussion.

\subsection{Dichotomy for an induced Finsler metric}
\label{sec:dichotomy}
Let $(Y^{2n+1},\alpha)$ be a contact manifold, let $\Lambda^{n}$ be a compact smooth manifold, and let $\mathscr{I}=\mathscr{I}(\Lambda)$ be the space of Legendrian immersions $i:\Lambda\to Y$ modulo reparametrization by diffeomorphisms. The notion of length of regular homotopies equips $\mathscr{I}$ with a Finsler pseudo-metric. More precisely, given two equivalence classes of immersions $[j_{0}]$, $[j_{1}]$, the pseudo-distance between them is:
\begin{equation*}  \mathrm{dist}_{\mathscr{I},\alpha}([j_{0}],[j_{1}])=\mathrm{inf}\set{\mathrm{length}_{\alpha}(i_{t}):i_{0}=j_{0}\circ \sigma\text{ and }i_{1}=j_{1}\circ \sigma'}
\end{equation*}
where $\sigma,\sigma'$ are diffeomorphisms of $\Lambda$. We should note that if $\sigma_{t}$ is an isotopy of diffeomorphisms then $\mathrm{length}_{\alpha}(i\circ \sigma_{t})=0$, and hence, if we want any hope of a non-degenerate metric, we need to at least mod out by diffeomorphisms isotopic to the identity. However, in order to prove our next result, we require modding out by all diffeomorphisms.

Let $\mathscr{E}\subset \mathscr{I}$ be the subset of embeddings, and let $\mathscr{E}=\mathscr{E}_{1}\cup \mathscr{E}_{2}\cup \dots$ be the decomposition into isotopy classes of embeddings. We have the following dichotomy:
\begin{theorem}\label{theorem:imm_dichotomy}
  For each isotopy class $\mathscr{E}_{k}$, the restriction of $\mathrm{dist}_{\mathscr{I},\alpha}$ to $\mathscr{E}_{k}$ is either identically zero or is non-degenerate.
\end{theorem}
Indeed, observe that for any ambient compactly-supported contactomorphism $\phi$, with $\phi^{*}\alpha=e^{g}\alpha$,
\begin{equation}\label{eq:quasi-invariant}
  \mathrm{dist}_{\mathscr{I},\alpha}(\phi(\Lambda_{0}),\phi(\Lambda_{1}))\le e^{\max g}  \mathrm{dist}_{\mathscr{I},\alpha}(\Lambda_{0},\Lambda_{1}),
\end{equation}
and then Theorem \ref{theorem:imm_dichotomy} follows from:
\begin{lemma}\label{lemma:chekanov_leg}
  Let $\mathscr{E}_{k}$ be an isotopy class of compact Legendrians in $(Y,\xi)$, and suppose that $\delta:\mathscr{E}_{k}\times \mathscr{E}_{k}\to [0,\infty)$ is a pseudo-metric so that for each compactly supported contactomorphism $\phi$ there is $C(\phi)$ so that $\delta(\phi(\Lambda_{0}),\phi(\Lambda_{1}))\le C(\phi)\delta(\Lambda_{0},\Lambda_{1}).$ Then either $\delta=0$ or $\delta$ is non-degenerate.
\end{lemma}
The proof of Lemma \ref{lemma:chekanov_leg} is exactly the argument used in \cite[Theorem 2]{chekanov_2000}. See also \cite{zhang_rosen}. This dichotomy suggests the following question:
\begin{question}
  Is $\mathrm{dist}_{\alpha,\mathscr{I}}$ non-degenerate on any isotopy class $\mathscr{E}_{k}$?
\end{question}

\subsection{The oscillation norm and translated points}
\label{sec:translated points}

Shelukhin relates the oscillation norm $\abs{\varphi}^{\mathrm{osc}}_{\alpha}$ to the existence of translated points of the contactomorphism $\varphi$. By a cut-off argument relating contactomorphisms of $Y$ to compactly supported Hamiltonian diffeomorphisms of $SY$, \cite[Theorem B]{shelukhin_contactomorphism} uses earlier results of \cite{leaf_wise_albers_frauenfelder} to conclude:
\begin{equation}\label{eq:shelukhin_conj}
  \abs{\varphi}^{\mathrm{osc}}_{\alpha}<2(\text{spectral gap of $\alpha$})\implies
  \begin{minipage}[c][1.4cm][c]{0.5\linewidth}
    $\varphi$ has at least one translated point, with at least $\dim H_{*}(Y,\Z/2)$ in the case all the translated points are non-degenerate.
  \end{minipage}
\end{equation}
Here the \emph{spectral gap} is the minimal positive action of a Reeb orbit of $(Y,\alpha)$. Because of the Rabinowitz Floer homology used in \cite{leaf_wise_albers_frauenfelder}, Shelukhin's argument assumes that $Y$ admits an exact filling. His result leads naturally to the following conjecture:
\begin{conj}[Conjecture 31 in \cite{shelukhin_contactomorphism}]\label{conj:shelukhin_conj}
  If $\varphi$ is a contactomorphism of an arbitrary compact contact manifold $Y$, then the implication \eqref{eq:shelukhin_conj} holds.
\end{conj}
% We note that this is a weakened version of Sandon's conjecture, see \cite{sandon_13}, which was shown to be false in \cite{cant_sandon_conj}. 

The preprint \cite{oh_legendrian_entanglement} suggests a way to resolve Shelukhin's conjecture via a \emph{Legendrianization} of the problem. % with a more restrictive hypothesis (by a factor of two):
% \begin{prop}[Theorem 1.23 in \cite{oh_legendrian_entanglement}]\label{prop:oh_1}
%   If $\varphi$ is a contactomorphism of an arbitrary compact contact manifold $Y$, then $\abs{\varphi}^{\mathrm{osc}}_{\alpha}<(\text{spectral gap of $\alpha$})$ implies the conclusion of \eqref{eq:shelukhin_conj}.
% \end{prop}

The crux of Oh's argument is to relate the translated points of a contactomorphism to the Reeb chords of a certain Legendrian. We briefly explain the construction. Recall that any contactomorphism lifts to an equivariant symplectomorphism of $SY$, and hence defines an $\R$-invariant Lagrangian inside $SY\times SY$ with the exact symplectic form $\d \lambda_{0}-\d \lambda_{1}$. It is well-known that $\R$-invariant Lagrangians in an exact symplectic manifold project to Legendrians after quotienting by the Liouville flow. In this case the quotient $Q(Y)=(SY\times SY)/\R$ is a well-defined contact manifold (as the Liouville flow is free and proper) and so the Lagrangian graph of an equivariant lift of a contactomorphism $\varphi$ projects to Legendrian $\Lambda(\varphi)\subset Q(Y)$. This procedure of relating contactomorphisms to Legendrians was previously considered in \cite[\S 5]{lychagin_77}, \cite{chekanovQF}, \cite{bhupal}, \cite{sandon_notes}, \cite{sandon_12}, \cite{colin_sandon} \cite{guogang_liu_positive_loops}.

It is not hard to see that a choice of contact form $\alpha_{0}$ on $Y$ establishes a contactomorphism between $Q(Y)$ and the contactization $(Y\times SY,\alpha_{0}-\lambda_{1})$ described above in \S\ref{sec:non-LDR}. Moreover, it is straightforward to verify that, for this contact form, the translated points of $\varphi$ are in natural bijection with Reeb chords between $\Lambda(\id)$ and $\Lambda(\varphi)$.

This perhaps suggests that Shelukhin's conjecture might follow if one could prove that $Q(Y)$ was LDR. However, our Theorem \ref{theorem:nonLDR} shows that these contactizations are not LDR for certain exotic $Y$.

\subsubsection*{Acknowledgements}
\label{sec:ack}
I want to thank Egor Shelukhin, Octav Cornea, Jakob Hedicke, and Filip Broci{\'c} for useful discussions during the writing of this paper. This work was completed at the University of Montreal with funding from the CIRGET research group.

\section{Proofs of Theorems}
\label{sec:proofs}

Theorems \ref{theorem:mohnke}, \ref{theorem:nonLDR}, \ref{theorem:generating_function}, \ref{theorem:immersion_disjoining}, and \ref{theorem:imm_dichotomy} are proved in \S\ref{sec:theorem_mohnke}, \S\ref{sec:proof_nonLDR}, \S\ref{sec:proof_gf}, \S\ref{sec:proof_immersion}, and \S\ref{sec:imm_dichotomy_proof}, respectively. Proposition \ref{prop:rizell_sullivan} is proved in \S\ref{sec:proof_conjecture}.

\subsection{On the strict LDR property for quasi-fillable contact manifolds}
\label{sec:theorem_mohnke}
Recall the following \emph{Lagrangianization} of a Legendrian $\Lambda\subset (Y,\alpha)$ due to \cite{klaus_chord}. Fix a small $\delta>0$, and, for $k>1$, consider the ``rectangle''
\begin{equation*}
  \Phi:(s,t,x)\in [0,\log(k)]\times [0,T]\times \Lambda\mapsto e^{s}\rho_{t}(x)\in SY,
\end{equation*}
where $\rho_{t}$ is the time $t$ Reeb flow and multiplication by $e^{s}$ denotes the time $s$ Liouville flow in $SY$. Consider $\Lambda$ as sitting in the copy of $Y$ in $SY$ corresponding to the choice of contact form $\alpha$. As a consequence $\Phi^{*}\lambda=e^{s}\d t$.

The restriction of $\Phi$ to the boundary of the rectangle is called a \emph{Lagrangian tetragon} in \cite{tetragons}. We consider a smoothed version of these tetragons, as in \cite{klaus_chord}.

Let $L_{k,\delta}(\Lambda)$ be a Lagrangian obtained by restricting $\Phi$ to a subset of the form $\Gamma\times \Lambda$ where $\Gamma\subset [0,\log(k)]\times [0,T]$ is an embedded loop so $\int_{\Gamma} e^{s}\d t=e^{-\delta}(k-1)T$ (the area of $[0,\log(k)]\times [0,T]$ with the symplectic form $e^{s}\d s\wedge \d t$ is $(k-1)T$). Note that $L_{k,\delta}(\Lambda)$ is embedded if $T$ is smaller than the action of all Reeb chords of $\Lambda$. Of course, if the Reeb flow is not complete, then we may be forced to pick $T$ very small, irrespective of the Reeb chords of $\Lambda$.

The key observation is the following:
\begin{lemma}\label{lemma:key_1}
  If $\phi:Y\to Y$ is a strict contactomorphism which disjoins $\Lambda$, i.e., no chords between $\Lambda$ to $\phi(\Lambda)$, then any lift of $\phi$ to $\phi':SY\to SY$ displaces $L_{k,\delta}(\Lambda)$.
\end{lemma}
\begin{proof}
  If $\phi'(L_{k,\delta}(\Lambda))\cap L_{k,\delta}(\Lambda)$ is non-empty, then we can find $(s,t,x)$ and $(s',t',x')$ so that $\phi'(e^{s}\rho_{t}(x))=e^{s'}\rho_{t'}(x').$ Since $\phi'$ is a lift of $\phi$, we conclude that $\phi(\rho_{t}(x))=\rho_{t'}(x')$. The strictness of $\phi$ implies that $\phi\rho_{t}=\rho_{t}\phi$, thus $\phi(x)=\rho_{t'-t}(x')$. But this contradicts the hypothesis that $\phi$ disjoins $\Lambda$. This completes the proof.
\end{proof}

In general, contactomorphisms of $Y$ have unique lifts to equivariant symplectomorphisms of $SY$, and the lift of a contact isotopy is a Hamiltonian isotopy. One of the results of \cite{shelukhin_contactomorphism} is a useful cut-off theorem:
\begin{prop}\label{prop:shelukhin}
  Given a compactly supported contact isotopy $\phi_{t}$ with oscillation energy $E$, so that $H_{t}^{-1}(0)\ne \emptyset$ for each $t$, there is a compactly supported Hamiltonian isotopy $\psi_{t}:SY\to SY$ so that:
  \begin{enumerate}
  \item $\psi_{1}(L_{k,\delta}(\Lambda))=\phi_{1}'(L_{k,\delta}(\Lambda))$, where $\phi_{1}'$ is the equivariant lift of $\phi_{1}$.
  \item the (Hofer) oscillation energy of $\psi_{t}$ is bounded by $e^{\delta}kE$.
  \end{enumerate}
\end{prop}
\begin{proof}
  This follows from \cite[Proposition 41 and Lemma 42]{shelukhin_contactomorphism}. 
\end{proof}

Next, recall Chekanov's result on the lower bound for the displacement energy of a compact Lagrangian in a tame symplectic manifold:
\begin{prop}[\cite{chekanov_1998}]
  Let $L\subset X$ be a compact Lagrangian in a tame symplectic manifold $X$, and let $\sigma(X,L)$ be the infimum of positive symplectic areas of disks with boundary on $L$. Then any compactly supported Hamiltonian isotopy displacing $L$ has Hofer oscillation energy at least $\sigma(X,L)$.
\end{prop}

Theorem \ref{theorem:mohnke} is proved as follows. First assume $Y$ is aspherically quasi-fillable and has a complete Reeb flow. In search of a contradiction, suppose that $\phi_{t}$ is a compactly supported contact isotopy\footnote{If $Y$ is compact, we can change $\phi_{t}$ to $\rho_{st}\circ \phi_{t}$, where $\rho$ is the Reeb flow and $s\in \R$, in order to achieve that $H_{t}^{-1}(0)\ne \emptyset$ for each $t$.} so that $\phi_{1}$ is strict and disjoins $\Lambda$, with oscillation energy $E$ smaller than the action of the shortest Reeb chord of $\Lambda$. Then we can pick $T>E$ in the definition of $L_{k,\delta}(\Lambda)$. 

The minimal positive symplectic area of a disk with boundary on $L_{k,\delta}(\Lambda)$ in $SY$ is $e^{-\delta}(k-1)T$ (this argument is also used in \cite{klaus_chord}). The aspherically quasi-fillable assumption implies an embedding $SY\to X$ into a tame symplectic manifold which satisfies $\pi_{2}(X,SY)=0$. It follows that $\sigma(X,L_{k,\delta}(\Lambda))=e^{\delta}(k-1)T$. On the other hand, Proposition \ref{prop:shelukhin} and Lemma \ref{lemma:key_1} guarantee a Hamiltonian isotopy $\psi_{t}$ so that $\psi_{1}$ displaces $L_{k,\delta}(\Lambda)$ and $\psi_{t}$ has oscillation energy at most $e^{\delta}kE$. Since $E<T$, for $k$ large enough and $\delta$ small enough, we have $e^{\delta}kE<e^{-\delta}(k-1)T$, contradicting Chekanov's result. This contradiction completes the proof of Theorem \ref{theorem:mohnke}.

In the case when $Y$ is not assumed to have a complete Reeb flow, or be aspherically fillable, then proceed as above, and take any of the Lagrangians $L_{k,\delta}(\Lambda)$ for a small $T$. There is a contradiction if $e^{\delta}kE<\sigma(X,L_{k,\delta}(\Lambda))$. Thus we can take:
\begin{equation*}
  C(\Lambda,\alpha)=\frac{\sigma(X,L_{k,\delta}(\Lambda))}{e^{\delta}k}.
\end{equation*}
This constant may be non-optimal, but it depends only on $\Lambda=\Lambda_{0}$ (and some auxiliary choices) and not on any disjoining isotopy $\Lambda_{t}$, as desired.

\subsubsection{Comments on the argument}
\label{sec:comments}

Note that only $\phi_{1}$ is required to be strict, rather than the whole isotopy $\phi_{t}$ being strict. Moreover, we only require that:
\begin{equation}\label{eq:spec_1}
s\in [0,T]\implies \phi_{1}(\rho_{s}(\Lambda))=\rho_{s'}(\phi_{1}(\Lambda))
\end{equation}
for some $s'(s)$.

Given any isotopy $\Lambda_{t}$, there is an extension $\phi_{t}$ so that $\phi_{1}(\rho_{s}(\Lambda))=\rho_{s}(\phi_{1}(\Lambda))$ for $s\in [0,\epsilon]$, where $\epsilon$ depends on the \emph{shortest chord of $\Lambda_{t}$, $t\in [0,1]$}. By preconcatenating $\phi_{t}$ with an isotopy which ``compresses'' $\rho_{[0,T]}(\Lambda)$ into $\rho_{[0,\epsilon]}(\Lambda)$, \eqref{eq:spec_1} can be achieved. However, these compression isotopies appear to require oscillation energy at least $T-\epsilon$. This seems to be right on the borderline and the author was unable to use this to prove that aspherically quasi-fillable contact manifolds are (non-strict) LDR.

\subsection{Proof of Theorem \ref{theorem:nonLDR}}
\label{sec:proof_nonLDR}

The construction has a few steps, broken up into subsections. Let $Y_{0}^{2n+1},Y_{1}^{2m+1}$ be contact manifolds, with a choice of form $\alpha_{0}$ for $Y_{0}$, and let $Q=Y_{0}\times SY_{1}$ be considered as a contact manifold with form $A=\alpha_{0}-\lambda_{1}$. For later use, observe that the Reeb vector field for $(Q,A)$ is just the Reeb field for $(Y_{0},\alpha_{0})$, i.e., is tangent to the fibers of the projection $Q\to SY_{1}$. 

% Suppose that $\dim(Y_{0})=2n+1$ and $\dim(Y_{1})=2m+1$. Then $\dim(Q)=2(n+m+1)+1$, and:
% \begin{equation*}
%   A\wedge \d A^{n+m+1}=(-1)^{m+1}{\textstyle{n+m+1\choose n}}\alpha_{0}\wedge \d\alpha_{0}^{n}\wedge \d\lambda_{1}^{m+1}=\text{a volume form},
% \end{equation*}
% This proves that $(Q,A)$ is indeed contact.

\subsubsection{Lifting exact Lagrangians in $SY_1$ to Legendrians in $Q$.}
\label{sec:lifting}
Suppose $j:L\to SY_{1}$ is an exact Lagrangian embedding, in the sense that $j^{*}\lambda_{1}=\d f$ for a smooth function $f:L\to \R$. The following geometric construction, depending on an auxiliary choice of Legendrian in $Y_{0}$, lifts exact Lagrangians in $SY_{1}$ to Legendrians in $Q$.

\begin{lemma}
  Let $L,j,f$ be as above. Pick a Legendrian embedding $i:\Lambda_{0}\to Y_{0}$, and let $\rho_{t}:Y_{0}\to Y_{0}$ be the time $t$ $\alpha_{0}$-Reeb flow. Then:
  \begin{equation}\label{eq:formula_legendrian}
    J:(x,y)\in \Lambda_{0}\times L \mapsto (\rho_{f(y)}(i(x)),j(y))\in Y_{0}\times SY_{1}=Q
  \end{equation}
  defines a Legendrian embedding into $(Q,A)$.
\end{lemma}
\begin{proof}
  Left to the reader. 
\end{proof}

\subsubsection{Lifting Hamiltonian isotopies of $SY_1$ to contact isotopies of $Q$}
\label{sec:lifting_ham}

Let $X_{t}$ be a compactly supported time dependent Hamiltonian vector field on $SY_{1}$, and define the Hamiltonian by the formula $X_{t}\intprod \d\lambda_{1}=-\d H_{t}$. Let $\varphi_{t}$ be the time $t$ flow of this non-autonomous system.

\begin{lemma}
  Let $K_{t}:Q\to \R$ be the lift of $-H_{t}$, i.e., $K_{t}=-H_{t}\circ \pr_{1}$. Then the unique contact vector field $N_{t}$ on $Q$ satisfying $N_{t}\intprod A=K_{t}$ generates an $A$-strict contact isotopy lifting $\varphi_{t}$.
\end{lemma}
\begin{proof}
  Recall that $N_{t}$ is contact if and only if:
  \begin{equation}\label{eq:contact-vf}
    \d(N_{t}\intprod A)+N_{t}\intprod \d A=\beta_{t} A,
  \end{equation}
  and $N_{t}$ induces a flow by \emph{strict} contactomorphisms if $\beta_{t}=0$ for all $t$. Since $N_{t}\intprod A$ is pulled back from $SY_{1}$, and the Reeb field for $A$ is tangent to the fibers of $Q\to SY_{1}$, we conclude that $\beta_{t}=0$.

  We claim that:
  \begin{equation}\label{eq:formula_lifted_vector}
    N_{t}=(X_{t}\intprod \lambda_{1}-H_{t})R+X_{t},
  \end{equation}
  where $R$ is the Reeb flow for $Y_{0}$ (and $Q$). Indeed, this satisfies $N_{t}\intprod A=-H_{t}$, and \eqref{eq:contact-vf} follows immediately from $X_{t}\intprod \d\lambda_{t}+\d H_{t}=0$.
  
  Since $R$ is tangent to the fibers of $Q\to SY_{1}$, the flow of $N_{t}$ is a lift of the flow of $X_{t}$. This completes the proof.
\end{proof}

\subsubsection{Lifting Lagrangian isotopies}
\label{sec:lifting_primitive}
For later use, it will be important to answer the following question: if $\Lambda$ is a Legendrian lift of $L$, as constructed in \S\ref{sec:lifting}, $\varphi_{t}$ is a Hamiltonian isotopy, and $\Phi_{t}$ is the lifted contact isoptopy as above, what is the relationship between $\Phi_{1}(\Lambda)$ and $\varphi_{1}(L)$? This is answered with:

\begin{lemma}\label{lemma:lifting_primitive}
  Let $\varphi_{t}$ be the Hamiltonian isotopy of $SY_{1}$ defined by $H_{t}$, and let $\Phi_{t}$ be the lifted contact isotopy with contact Hamiltonian $-H_{t}\circ \pr_{1}$. Suppose that $J$ is the Legendrian lift of  an exact Lagrangian $j:L\to SY_{1}$, using the primitive $f:L\to \R$ and auxiliary Legendrian $i:\Lambda_{0}\to Y_{0}$, as defined in \eqref{eq:formula_legendrian}. Then $\Phi_{t}\circ J$ is a lift of $\varphi_{t}\circ j$ of the form \eqref{eq:formula_legendrian}, using the same auxiliary Legendrian $i$ and the new primitive
  \begin{equation*}
    f_{t}(y)=f(y)+\int_{0}^{t}(X_{\tau}\intprod \lambda_{1}-H_{\tau})\circ \varphi_{\tau}(j(y))\,\d \tau,
  \end{equation*}
  where $X_{t}$ is the infinitesimal generator of $\varphi_{t}$.  
\end{lemma}
\begin{proof}
  Use Cartan's magic formula to deduce that:
  \begin{equation*}
    \varphi_{t}^{*}\lambda_{1}=\lambda_{1}+\d\int_{0}^{t}(X_{\tau}\intprod \lambda_{1}-H_{\tau})\circ \varphi_{\tau}\, \d \tau,
  \end{equation*}
  and thus $f_{t}$ is a valid primitive for $\varphi_{t}\circ j$. Moreover, if we construct
  \begin{equation*}
    J_{t}(x,y)=(\rho_{f_{t}(y)}(i(x)),\varphi_{t}\circ j (y)),
  \end{equation*}
  then differentiating with respect to $t$ yields:
  \begin{equation*}
    \pd{}{t}J_{t} = ((X_{t}\intprod \lambda_{1}-H_{t}) R+X_{t})\circ J_{t}=N_{t}\circ J_{t},
  \end{equation*}
  where we have used \eqref{eq:formula_lifted_vector}. Since $N_{t}$ generates $\Phi_{t}$ we conclude $J_{t}=\Phi_{t}\circ J$, as desired.
\end{proof}

\subsubsection{Compact exact Lagrangians in exotic symplectizations}
\label{sec:exotic}
Let $Y_{1}$ be \emph{exotic} in the sense that $SY_{1}$ contains a compact exact Lagrangian $L$, and write $\lambda_{1}|_{L}=\d f$. According to \cite{murphy_closed_exact} there exist such compact $Y_{1}$ with $\dim(Y_{1})\ge 5$. See \S\ref{sec:non-LDR} for more details.

Let $e^{s}L$ be the translated Lagrangian in $SY_{1}$ (i.e., flow $L$ by time $s$ using the Liouville flow). Let $F_{s}(x)=e^{s}x$ for $x\in L$, and observe that $F_{s}^{*}\lambda_{1}=e^{s}F_{0}^{*}\lambda_{1}=e^{s}\d f$. Thus $F_{s}$ is an \emph{exact isotopy}, and hence by the isotopy extension theorem, $F_{s}$ can be extended to a Hamiltonian isotopy of $SY_{1}$.

\begin{lemma}[due to Sikorav, see \cite{chekanov_2000}]
  Suppose that $s$ is large enough that $e^{s}L$ and $L$ are disjoint. Then, for any $\epsilon>0$, there exist compactly supported Hamiltonians $H_{t}$ with
  \begin{equation*}
    \text{oscillation energy}=\int_{0}^{1}(\max H_{t}-\min H_{t})\d t<\epsilon,
  \end{equation*}
  and the time $1$ flow induced by $H_{t}$ takes $L$ to $e^{s}L$.
\end{lemma}
In particular, the Chekanov-Hofer distance on the space of Lagrangians is degenerate in this exotic symplectization.
\begin{proof}
  The argument is rather non-constructive. Consider $K=L\cup e^{s}L$ as an exact Lagrangian, and observe that $r\mapsto e^{r}K$ is an exact isotopy. Thus we can find a Hamiltonian diffeomorphism simultaneously taking $L$ to $e^{r}L$ and $e^{s}L$ to $e^{s+r}L$. Thus, by conjugation invariance of the Hofer norm, if we can go from $L$ to $e^{s}L$ with oscillation energy $E$, then we can go from $e^{r}L$ to $e^{s+r}L$ with oscillation energy $E$.

  The key observation is that if we can go from $e^{r}L$ to $e^{s+r}L$ with oscillation energy $E$, \emph{then we can go from $L$ to $e^{s}L$ with oscillation energy $e^{-r}E$}. Indeed, given a Hamiltonian isotopy $\varphi_{t}$ taking $e^{r}L$ to $e^{r+s}L$ with oscillation energy $E$, we simply conjugate with the expanding map $\psi:x\mapsto e^{r}x$, i.e., $\psi^{-1}\circ \varphi_{t}\circ \psi$ takes $L$ to $e^{s}L$ with oscillation energy $e^{-r}E$. Combining this with the previous paragraph's observation and taking $r\to\infty$ implies the desired result. 
\end{proof}

\subsubsection{Disjoining Legendrians with arbitrarily small energy}
\label{sec:disjoining}

Fix $\epsilon>0$, and let $Y_{1}$ be exotic so that there is a compact exact Lagrangian $L$ in $SY_{1}$. Using \S\ref{sec:lifting}, construct a Legendrian lift $\Lambda\subset Q$ of $L$. By \S\ref{sec:exotic}, there is a compactly supported Hamiltonian isotopy of $SY_{1}$ of oscillation energy at most $\epsilon$ which displaces $L$. Using the results of \S\ref{sec:lifting_ham}, this lifts to a contact isotopy of $Q$. Moreover, since the contact Hamiltonian is\footnote{If $Y_{0}$ is non-compact then we need to cut-off but this will not change the oscillation energy.} the lift of the Hamiltonian from $SY_{1}$, the oscillation energy of the lift is at most $\epsilon$. The lifted isotopy disjoins $\Lambda$ from itself. This completes the proof of the first part of Theorem \ref{theorem:nonLDR}.

\subsubsection{Degeneracy of the pseudo-metric}
\label{sec:degeneracy}
To prove the pseudo-metric is degenerate, we will find two distinct Legendrians $\Lambda_{1},\Lambda_{2}$, and isotopies of arbitrarily short $A$-length joining $\Lambda_{1},\Lambda_{2}$. We continue to work in $(Q,A)$ where $Q=Y_{0}\times SY_{1}$ for exotic $Y_{1}$, and suppose $j:L\to SY_{1}$ is an exact Lagrangian with primitive $f$. Again pick $s>0$ so $L,e^{s}L$ are disjoint.

Naively, $\Lambda_{1}$ will be a Legendrian lift of $L$, $\Lambda_{2}$ will be a lift of $e^{s}L$, and the isotopies will be lifts of Hamiltonian isotopies taking $L$ to $e^{r}L$. There is a slight technical point: if $\varphi_{t}^{k}$ is a sequence of Hamiltonian isotopies taking $L$ to $e^{s}L$ with oscillation energy tending to zero, and $\Phi_{t}^{k}$ are the lifts to contact isotopies, then $\Phi_{1}^{k}(\Lambda_{1})$ may be \emph{distinct} lifts of $e^{s}L$, and thus there is no \emph{fixed} $\Lambda_{2}$ we can use to establish the degeneracy of the pseudo-metric. To resolve this, we analyze how the primitives of $\varphi_{t}^{k}\circ j$ change along the isotopies generated by Sikorav's argument.

First we prove the following, slightly ad hoc, result, and in the next section we'll show that its hypotheses can be achieved with Sikorav's argument.
\begin{lemma}\label{lemma:ad_hoc}
  Let $j:L\to SY_{1}$ be an exact Lagrangian, as above. Suppose that $\varphi_{t}^{k}$, $k=1,2,\dots$ is a sequence of Hamiltonian isotopies so that $\varphi_{1}^{k}(j(y))=e^{s}j(y)$ (i.e., $\varphi_{1}^{k}$ is fixed on the image of $L$, and agrees with the rescaling by $e^{s}$ diffeomorphism). Let $H_{t}^{k}$ and $X_{t}^{k}$ be the (compactly supported) Hamiltonian and vector field for $\varphi_{t}^{k}$. Suppose that:
  \begin{equation*}
    \sup_{k}\sup_{y\in L}\abs{\int_{0}^{1}(X_{t}^{k}\intprod \lambda_{1}-H_{t}^{k})\circ \varphi^{k}_{t}(y)\,\d t}<\infty
  \end{equation*}
  and the oscillation energy $\int_{0}^{1}\max H_{t}^{k}-\min H_{t}^{k}\,\d t$ converges to zero as $k\to\infty$. Then there exist fixed Legendrian lifts $\Lambda_{1},\Lambda_{2}$ in $Q$ of $L$, $e^{s}L$, which can be joined by isotopies of arbitrarily small $A$-length.
\end{lemma}
\begin{proof}
  Suppose $\Lambda_{1}$ is described by $J$ from \eqref{eq:formula_legendrian}. Recall the result of Lemma \ref{lemma:lifting_primitive}: if $\Phi^{k}_{t}$ is the lift of $\varphi^{k}_{t}$, then $\Phi^{k}_{1}\circ J$ is also given by \eqref{eq:formula_legendrian}, using $\varphi^{k}_{1}\circ j$ and the primitive:
  \begin{equation*}
    f_{k}(y)=f(y)+\int_{0}^{1}(X_{t}^{k}\intprod \lambda_{1}-H_{t}^{k})\circ \varphi_{t}^{k}(y)\,\d t.
  \end{equation*}
  Since $\varphi^{k}_{1}\circ j=\varphi^{1}_{1}\circ j$, by assumption, $f_{k}(y)-f_{1}(y)=C_{k}$
  is a constant. Referring to \eqref{eq:formula_legendrian}, conclude that $\Phi_{1}^{k}\circ J$ and $\Phi_{1}^{1}\circ J$ \emph{differ by a Reeb flow of time $C_{k}$}.

  By assumption, $C_{k}$ is bounded in $k$. In particular, there exists a subsequence which converges, say to $C$. Let $\Lambda_{2}$ be the time $C$ Reeb flow of $\Phi^{1}_{1}\circ J$. The preceding observation implies that $\Phi^{k}_{1}\circ J$ differs from a time-$\epsilon_{k}$ Reeb flow from $\Lambda_{2}$ where $\epsilon_{k}\to 0$; therefore $\Phi^{k}_{1}\circ J$ can be be joined to $\Lambda_{2}$ by an isotopy of $A$-length $\epsilon_{k}$ (take a cut-off of the time-$\epsilon_{k}$ Reeb flow).

  The result now follows: since the isotopy $\Phi^{k}_{t}\circ J$ has arbitrarily small oscillation energy as $k\to\infty$, it also has $A$-length converging to zero (note that the oscillation energy bounds the $A$-length since the contact Hamiltonian takes the value $0$). Concatenate with the cut-off of the time-$\epsilon_{k}$ Reeb flow to join $\Phi^{k}_{1}\circ J$ to the fixed Legendrian $\Lambda_{2}$. The result joins $\Lambda_{1}$ to $\Lambda_{2}$ with arbitrarily small $A$-length as $k\to\infty$.
\end{proof}

\subsubsection{Achieving the bound on the primitive}
\label{sec:achieving_boundedness}

In this section, we analyze Sikorav's argument more carefully and show that the hypotheses of Lemma \ref{lemma:ad_hoc} can be achieved.

Let $j':L\sqcup L\to SY_{1}$ be the coproduct of $j$ and $e^{s}j$ (so the image is $L\cup e^{s}L$), as in \S\ref{sec:exotic}. For $r\in [0,1]$ consider the Lagrangian isotopy $j'_{r}(y)=e^{r}j'(y)$; this isotopy slides $L\cup e^{s}L$ up to $e^{1}L\cup e^{s+1}L$ in time $1$.

The Lagrangian isotopy extension theorem implies there is an ambient Hamiltonian isotopy $g_{r}$ so that $g_{r}(y)=j_{r}'(y)$ for $y\in L\cup e^{s}L$.

Let $\psi_{r}:SY_{1}\to SY_{1}$ be the global expanding map $\psi_{r}(z)=e^{r}z$, so $\kappa_{r}=g_{r}^{-1}\circ \psi_{r}$ satifies:
\begin{enumerate}
\item $\kappa_{1}^{*}\omega=e^{1}\omega$,
\item $\kappa_{r}(y)=y$ for $y\in L\cup e^{s}L$.
\end{enumerate}
Reparametrize $r=r(t)$ via a map $[0,1]\to [0,1]$, where $r'(t)=0$ for $t$ near $0$ and $1$, rewrite $\kappa_{t}=\kappa_{r(t)}$, and smoothly extend $\kappa$ to all $t>0$ by $\kappa_{t+1}=\kappa_{t}\kappa_{1}$.

Now let $\varphi_{t}$ be any Hamiltonian isotopy taking $L$ to $e^{s}L$ in time $1$, and define
\begin{equation*}
  \varphi^{k}_{t}=\kappa_{k}^{-1}\circ \varphi_{t}\circ\kappa_{k}.
\end{equation*}
We claim that $\varphi^{k}_{t}$ satisfies the hypotheses of Lemma \ref{lemma:ad_hoc}.

If $X_{t}$ is the generating vector field of $\varphi_{t}$, then the generator of $\varphi^{k}_{t}$ is
\begin{equation*}
  X_{t}^{k}=\d \kappa_{k}^{-1}\circ X_{t}\circ \kappa_{k},
\end{equation*}
and inserting this into $\omega$ we obtain:
\begin{equation*}
  X_{t}^{k}\intprod \omega=e^{-k}X_{t}^{k}\intprod \kappa_{k}^{*}\omega=e^{-k}\kappa_{k}^{*}(X_{t}\intprod \omega)=\d(e^{-k}H_{t}\circ \kappa_{k}),
\end{equation*}
where $H_{t}$ is the Hamiltonian for $\varphi_{t}$. In particular, $\varphi_{t}^{k}$ is Hamiltonian, and, as $k\to\infty$, the oscillation energy of $\varphi^{k}_{t}$ tends to zero.

It remains only to prove the bound on $\int_{0}^{1}(X_{t}^{k}\intprod \lambda_{1}-H^{k}_{t})\circ \varphi_{t}^{k}(y)\,\d t$. To do so, introduce the notation $\gamma^{k}_{y}(t)=\varphi_{t}^{k}(y)$, and consider $\gamma^{k}_{y}$ as a path which starts on $L$ and ends on $e^{s}L$. Note that as $k$ varies, $\gamma^{k}_{y}$ varies relative its endpoints, i.e., $\gamma^{k}_{y}(0)=y$ and $\gamma^{k}_{y}(1)=e^{s}y$. Recognizing $\ud{}{t}\gamma^{k}_{y}(t)=X_{t}^{k}\circ \varphi_{t}^{k}(y),$ conclude:
\begin{equation*}
  \int_{0}^{1}(X_{t}^{k}\intprod \lambda_{1}-H^{k}_{t})\circ \varphi_{t}^{k}(y)\,\d t=\int_{0}^{1} (\gamma_{y}^{k})^{*}\lambda_{1}-H_{t}^{k}\circ \gamma_{y}^{k}(t)\,\d t=:a_{k}(y).
\end{equation*}
This quantity is the Hamiltonian perturbed symplectic action for paths. We claim that:
\begin{equation*}
  \pd{a_{k}}{k}(y)=\int_{0}^{1}\pd{H^{k}_{t}}{k}\circ \gamma^{k}_{y}(t).
\end{equation*}
Indeed, the other terms in the derivative of $a_{k}$ vanish because $\gamma_{y}^{k}$ is a critical point for the action with respect to variations in the path relative its endpoints. Referring to the formulas $H^{k}_{t}=e^{-k}H_{t}\circ \kappa_{k}$ and $\varphi^{k}_{t}=\kappa_{k}^{-1}\circ \varphi_{t}\circ \kappa_{k}$ to obtain:
\begin{equation*}
  \pd{a_{k}}{k}(y)=\int_{0}^{1}(-e^{-k}H_{t}+e^{-k}\d H_{t}\cdot \pd{\kappa_{k}}{k}\circ \kappa_{k}^{-1})\circ \varphi_{t}(y)\,\d t.
\end{equation*}
It is clear that:
\begin{equation}\label{eq:implication}
  \int_{0}^{\infty}\sup_{y\in L}\abs{\pd{a_{k}}{k}(y)}\d k<\infty\implies \sup_{k>0}\sup_{y\in L}\abs{a_{k}(y)}<\infty,
\end{equation}
and the right hand side is what we want to show. Because of the $e^{-k}$ factor, the left hand side of \eqref{eq:implication} will follow from:
\begin{equation}\label{eq:final_1}
  \sup_{k>0}\sup_{y\in L} \sup_{t\in [0,1]}\abs{\pd{\kappa_{k}}{k}\circ \kappa_{k}^{-1}}\circ \varphi_{t}(y)<\infty,
\end{equation}
as the other terms are bounded independently of $y,k$. However, by the time periodicity $\kappa_{k+1}=\kappa_{k}\kappa_{1}$, it suffices to bound the quantity in \eqref{eq:final_1} for $k\in [0,1]$. Since the union of $\varphi_{t}(L)$, $t\in [0,1]$ is contained in some compact subset of $SY_{1}$, we conclude \eqref{eq:final_1} holds for $k\in [0,1]$. Thus we can apply Lemma \ref{lemma:ad_hoc} to complete the proof of Theorem \ref{theorem:nonLDR}.

\subsection{Proof of Theorem \ref{theorem:generating_function}}
\label{sec:proof_gf}
Suppose that $\Lambda_{0}\subset J^{1}(B)$ is a Legendrian which admits a linear-at-infinity generating function $F:B\times \R^{N}\to \R$. Since linear functions have no critical points, $\Lambda_{0}$ is compact. Begin with a reduction: it suffices to prove Theorem \ref{theorem:generating_function} in the case when $B$ is compact. For any isotopy $\Lambda_{t}$, there is a compact subdomain with boundary $B_{0}\subset B$ so that $\Lambda_{t}\subset J^{1}(B_{0}^{\mathrm{int}})$ and $F(x,\eta)=\ell(\eta)$ is linear whenever $x\not\in B_{0}$. Doubling $B_{0}$ to a closed manifold completes the reduction. 

\subsubsection{Review of linear-at-infinity generating functions}
\label{sec:review_GF}
This section reviews the necessary results from the theory of generating functions.

\begin{lemma}\label{lemma:GF1}
  Let $B$ be a compact manifold, and suppose that $\Lambda_{0}\subset J^{1}(B)$ admits a linear-at-infinity generating function. For any Legendrian isotopy $\Lambda_{t}$, there is a smooth family $F_{t}$ of linear-at-infinity generating functions $F_{t}:B\times \R^{N}\to \R$ so that $F_{t}$ generates $\Lambda_{t}$. Moreover, we can arrange so that $F_{t}$ agrees with $F_{0}$ on the complement of a compact set.
\end{lemma}
\begin{proof}
  This is a combination of \cite[Theorem 3.1]{jordan_traynor} and \cite[Lemma 3.8]{sabloff_traynor_obstructions}. Begin with a linear-at-infinity function $F(x,\eta_{1})$ for $\Lambda_{0}$, and then define the stabilization:
  \begin{equation*}
    \bar{F}_{0}(x,\eta_{1},\eta_{2})=F(x,\eta_{1})+Q(\eta_{2})
  \end{equation*}
  where $Q$ is a quadratic form (depending on the isotopy $\Lambda_{t}$). For suitably chosen $Q$, there exists a smooth extension $\bar{F}_{t}(x,\eta_{1},\eta_{2})$ so that (i) $\bar{F}_{t}$ generates $\Lambda_{t}$, and (ii) $\bar{F}_{t}$ agrees with $\bar{F}_{0}$ on the complement of a compact set. The extension is constructed via Chekanov's formula, see \cite{chekanovQF,traynor_helix_links}. Let $\R^{N}$ be the space of possible pairs $(\eta_{1},\eta_{2})$. Applying \cite[Lemma 3.8]{sabloff_traynor_obstructions}, one concludes the existence of a fiber-preserving diffeomorphism $\varphi:B\times \R^{N}\to B\times \R^{N}$ so that $\bar{F}_{0}\circ \varphi$ is linear-at-infinity (rather than ``quadratic-linear-at-infinity''). Letting $F_{t}:=\bar{F}_{t}\circ \varphi$ completes the proof, as each $F_{t}$ will also be linear-at-infnity (and fiber-preserving diffeomorphisms do not change the generated Legendrian).
\end{proof}

The time-derivative of a family of generating functions is bounded by the contact Hamiltonian; see \cite[Lemma 3.1]{rizell_sullivan_persistence_1} for a related result:
\begin{lemma}\label{lemma:contact_ham_t_derivative}
  Suppose $F_{t}(x,\eta)$ is a family of generating functions inducing a Legendrian isotopy $\Lambda_{t}$. If $h_{t}$ is the contact Hamiltonian for the isotopy, $\Sigma_{t}$ is the fiberwise critical set of $F_{t}$, and $F_{t}'(x,\eta)$ is the time-derivative, then $$F_{t}'(x,\eta)\in [\min h_{t},\max h_{t}]$$ holds for all $(x,\eta)\in \Sigma_{t}$.
\end{lemma}
\begin{proof}
  Let $\varphi_{t,x}$ be a family of diffeomorphisms of the fiber space. It suffices to prove the result for $F_{t}(x,\varphi_{t,x}(\eta))$. The values of the contact Hamiltonian at time $t_{0}$ are the values of $\alpha_{\gamma(t_{0})}(\gamma'(t_{0}))$ for curves $\gamma(t)\in \Lambda_{t}$ (see \eqref{eq:contact_hamiltonian}). Pick a point $(x,\eta)$ in the fiberwise critical set $\Sigma_{t_{0}}$, and extend $(x(t),\eta(t))$ to a path for $t\in [t_{0},t_{0}+\epsilon)$. By picking $\varphi_{t,x}$ appropriately we may assume that $\eta(t)=\eta$ is independent of $t$.

  Let $j_{t}:B\times \R^{N}\to J^{1}(B)$ be the map returning the value and horizontal derivative (the restriction of $j_{t}$ to $\Sigma_{t}$ parametrizes $\Lambda_{t}$). It is easy to show that:
  \begin{equation*}
    j_{t}^{*}\alpha=\sum_{i}\pd{F_{t}}{\eta_{i}}\d\eta_{i}\text{ and }\alpha_{j_{t}(x,\eta)}(j_{t}'(x,\eta))=F'_{t}(x,\eta).
  \end{equation*}
  To compute the contact Hamiltonian, take the path $\gamma(t)=j_{t}(x(t),\eta)$. Then:
  \begin{equation*}
    \alpha_{\gamma(t)}(\gamma'(t))=F_{t}'(x(t),\eta)+(j_{t}^{*}\alpha)(\pd{x(t)}{t})=F_{t}'(x(t),\eta)+0.
  \end{equation*}
  Evaluating at $t=t_{0}$ shows that $F_{t_{0}}'(x,\eta)\in [\min h_{t_{0}},\max h_{t_{0}}]$, as desired.
\end{proof}

The next result concerns the fiberwise difference of two generating functions:
\begin{lemma}\label{lemma:GF2}
  Let $F_{1}(x,\eta_{1})$ and $F_{2}(x,\eta_{2})$ be two generating functions, and let $\Lambda_{1}$ and $\Lambda_{2}$ be the generated (immersed) Legendrians. Then the critical points of the fiberwise difference function:
  \begin{equation*}
    \Delta(x,\eta_{1},\eta_{2})=F_{1}(x,\eta_{1})-F_{2}(x,\eta_{2})
  \end{equation*}
  are in bijection with Reeb chords from $\Lambda_{1}$ to $\Lambda_{2}$ and the corresponding critical value is the action of the chord (including chords of zero and negative action). 
\end{lemma}
\begin{proof}
  This is a straightforward computation, recalling that Reeb chords are the intersection points of the Lagrangian projection, and the action is the difference in the $z$-values. See \cite[Proposition 3.1]{sabloff_traynor_obstructions} for the case when $F_{1}=F_{2}$.
\end{proof}

The next result shows that the Reeb chords from $\Lambda_{0}$ to $\Lambda_{t}$ are determined by a family of linear-at-infinity functions:
\begin{lemma}\label{lemma:main_3_prop}
  Let $\Lambda_{0}\subset J^{1}(B)$ be a Legendrian admitting a linear-at-infinity generating function and suppose that $\Lambda_{t}$ is an isotopy starting at $\Lambda_{0}$. There is a family of linear-at-infinity functions $K_{t}:B\times \R^{N}\to \R$ so that:
  \begin{enumerate}
  \item the critical points of $K_{t}$ are in bijection with the Reeb chords between $\Lambda_{0}$ and $\Lambda_{t}$ identifying the critical value with the action of the chord,
  \item the set of critical points of value $0$ of $K_{0}$ is a Morse-Bott submanifold diffeomorphic to $\Lambda_{0}$, and,
  \item if $h_{t}$ is the contact Hamiltonian for the isotopy, the time derivative $K_{t}'(x,\eta)$ lies in the interval $[\min h_{t},\max h_{t}]$ for all critical points $(x,\eta)$ of $K_{t}$.
  \end{enumerate}
\end{lemma}
\begin{proof}
  Let $F_{t}$ be a family of linear-at-infinity generating functions for $\Lambda_{t}$ as furnished by Lemma \ref{lemma:GF1}. The fiberwise difference function:
  \begin{equation*}
    \Delta_{t}(x,\eta_{1},\eta_{2})=F_{t}(x,\eta_{1})-F_{0}(x,\eta_{2}),
  \end{equation*}
  agrees with $\Delta_{0}$ whenever $\eta_{1}$ lies outside of some compact set, and satisfies (i) by Lemma~\ref{lemma:GF2}. It is well-known that there is an action gap for Reeb chords between $\Lambda_{0}$ and itself, i.e., chords of sufficiently small action are constant chords, and hence the set of critical points of value $0$ is diffeomorphic to $\Lambda_{0}$. This critical manifold is Morse-Bott because $\bd_{\eta}F_{0}(x,\eta)=0$ is cut out transversally, see \cite[Lemma 3.3]{sabloff_traynor_slice} for more details. Property (iii) follows from Lemma \ref{lemma:contact_ham_t_derivative}.

  Unfortunately, $\Delta_{t}$ is not linear-at-infinity. Indeed, if $F_{t}$ is defined on $B\times \R^{M}$, then $\Delta_{t}$ could fail to be linear on $B\times (C\times \R^{M}\cup  \R^{M}\times C)$ where $C\subset \R^{M}$ is a compact set. Let $K_{t}:=\Delta_{t}\circ \varphi_{t}$ where $\varphi_{t}$ is fiber preserving diffeomorphism of $B\times \R^{N}$ (and $N=2M$). By picking $\varphi_{t}$ appropriately, we can arrange that $K_{t}$ is linear at infinity. Properties (i), (ii), and (iii) are unaffected by such a modification. One uses a Moser-type deformation argument. The first step is to prove there is a fiberwise $\psi_{t}$ so that $\Delta_{t}\circ \psi_{t}=\Delta_{0}$ outside a compact set; this holds if and only if the generator $X_{t}$ of $\psi_{t}$ satisfies:
  \begin{equation*}
    X_{t}\intprod \d\Delta_{t}+\Delta_{t}'=0\text{ outside a compact set}.
  \end{equation*}
  This can be solved for $X_{t}$. More precisely, we can pick $X_{t}=-\Delta_{t}'V$ where $V$ is a constant vector field, pointing in the $\eta_{2}$ directions, satisfying $V\intprod \d\Delta_{t}=1$ whenever $\eta_{2}$ lies outside a compact set (bearing in mind that $\Delta_{t}'=0$ when $\eta_{1}$ lies outside a compact set). The second step is to prove there is a fiberwise $\kappa$ so that $\Delta_{0}\circ \kappa$ is linear-at-infinity. This can be constructed in a similar fashion by finding $\kappa_{s}$ so that
  \begin{equation*}
    [(1-s)\Delta_{0}(x,\eta_{1},\eta_{2})+s(\ell(\eta_{1})-\ell(\eta_{2}))]\circ \kappa_{s}=\Delta_{0}\text{ outside a compact set}.
  \end{equation*}
  Then $\varphi_{t}=\psi_{t}\circ \kappa_{1}^{-1}$ ensures that $\Delta_{t}\circ \varphi_{t}$ is linear-at-infinity. 
\end{proof}

\subsubsection{Singular Legendrian diagram associated to a family of functions}
\label{sec:singular}

To any a family of functions $K_{t}:B\times \R^{N}\to \R$ we associate a singular Legendrian diagram in $J^{1}([0,1])$ as follows: let $S_{t}$ be the set of critical points of $K_{t}$, and let $S$ be the union of $S_{t}$ over $t\in [0,1]$. There is a canonical map $S\to J^{1}([0,1])$ whose $z,q,p$ coordinates are given by:
\begin{equation*}
  z(t,x,\eta)=K_{t}(x,\eta)\hspace{1cm}q(t,x,\eta)=t\hspace{1cm}p(t,x,\eta)=K_{t}'(t,x,\eta).
\end{equation*}
Indeed, this map extends to the whole space $[0,1]\times B\times \R^{N}$ and the extension satisfies $\d z-p\d q=0$ on $S$. Thus, if $S$ is cut transversally, the canonical map parametrizes an immersed Legendrian. Of course, this is just a special case of the generating function framework. Let us call the image of the canonical map the \emph{singular Legendrian diagram} associated to the family $K_{t}$, and denote it by $\mathscr{D}$. The front of the singular Legendrian diagram is the Cerf diagram of the family $K_{t}$, see \cite[\S5]{cerf_stratification}.

\begin{lemma}\label{lemma:cerf_ctnty}
  Let $K:[0,1]\times B\times \R^{N}\to \R$ be a family of linear-at-infinity functions. Then for any open set $U$ around the singular Legendrian diagram, there is a $C^{1}$-neighborhood of $K$ whose elements have singular Legendrian diagrams in $U$.
\end{lemma}
\begin{proof}
  This is elementary, and is left to the reader. Bear in mind that the singular Legendrian diagram is compact because of the linear-at-infinity hypothesis.
\end{proof}

The rest of the proof can be summarized as follows. Lemma \ref{lemma:main_3_prop} gives a family of linear-at-infinity functions $K_{t}$ whose critical points are in bijection with Reeb chords from $\Lambda_{0}$ to $\Lambda_{t}$. The singular Legendrian diagram $\mathscr{D}$ satisfies the constraint that:
\begin{equation*}
  \mathscr{D}\cap \set{q=t}\subset \set{p\in [\min h_{t},\max h_{t}]}.
\end{equation*}
In words, there is a bound on the $p$ coordinates of the singular Legendrian diagram. In the coming sections we will show that, if $\int_{0}^{1}\max h_{t}-\min h_{t}\d t$ is smaller than the longest bar appearing in the \emph{barcode} of $K_{0}$, then $\mathscr{D}\cap \set{q=1}$ will be non-empty, and therefore there will be at least one Reeb chord between $\Lambda_{0}$ and $\Lambda_{1}$.

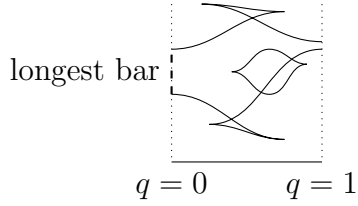
\begin{figure}[H]
  \centering
  \begin{tikzpicture}
    \draw (0,1.4) to[out=0,in=180] (1.5,0.8) to[out=180,in=0] (0.5,1) to[out=0,in=180] (2,2);
    \draw (0,2) to[out=0,in=180] (1.5,2.5) to[out=180,in=0] (0.4,2.6) to[out=0,in=180] (2,2.1);
    \draw[shift={(0.3,0.7)}] (0.5,1) to[out=0,in=180] (1,1.3) to[out=0,in=180] (1.5,1.1) to[out=180,in=0] (1,0.7) to[out=180,in=0] (0.5,1);
    \draw[dotted] (0,2.6)--(0,0.5) (2,2.6)--(2,0.5);
    \draw[thick,dashed] (0,1.4)--node[left]{longest bar}(0,2);
    \draw (0,0.5)node[below]{$q=0$}--(2,0.5)node[below]{$q=1$};
  \end{tikzpicture}
  \caption{The front projection of the singular Legendrian diagram $\mathscr{D}$ associated to a family of functions $K_t$. The $p$ coordinate is the slope of the front.}
  \label{fig:singular_leg_diagram}
\end{figure}

\subsubsection{Barcodes of linear-at-infinity functions}
\label{sec:barcodes}
Fix a family of linear-at-infinity functions $K_{t}$, defined on $B\times \R^{N}$, and suppose that $K_{t}(x,\eta)=\ell(\eta)$ outside some compact set for some $\ell\ne 0$, as in Lemma \ref{lemma:main_3_prop}. Choose $R$ large enough that $K_{t}(x,\eta)=\ell(\eta)$ when $K_{t}<-R$ and $K_{t}>R$. For each $t$, associate the \emph{persistence module} given by:
\begin{equation*}
  V_{t}(s)=H_{*}(K_{t}<s,K_{t}<-R;\Z/2).
\end{equation*}
The persistence maps $V_{t}(s_{1})\to V_{t}(s_{2})$ for $s_{1}\le s_{2}$ are induced by inclusions. It is clear that $V_{t}(s)=0$ for $s>R$, since the inclusion of $\ell<-R$ into $\ell<R$ is a deformation retract. 

The \emph{barcode} $B_{t}$ is the unique multiset of intervals $I_{1}^{t},\dots,I^{t}_{d}$ so that
\begin{equation*}
  V_{t}(s)=P(I_{1})\oplus \dots \oplus P(I_{d}),
\end{equation*}
where $P(I)$ is the trivial persistence module supported on the interval $I$ (equal to $\Z/2$ when $s\in I$, and $0$ otherwise). It is well-known that the endpoints of the bars in $B_{t}$ are critical values of $K_{t}$, and, if $K_{t}$ is Morse, then every critical value appears as an endpoint of a bar. Roughly speaking, we can think of $B_{t}$ as being a pairing of strands in the front of the singular Legendrian diagram $\mathscr{D}$; this is the perspective taken in \cite[\S12]{chekanov_pushkar}.

It is a fundamental fact that the barcode $B_{t}$ depends \emph{continuously} on $t$. This can be stated precisely by using the \emph{bottleneck distance}, see \cite{cohen_steiner_edelsbrunner_harer}. For our purposes, we require only the following statement. Let $L(t)$ be the length of the longest bar in $B_{t}$ (with $L(t)=0$ if there are no bars). Then for any $t_{0}\in (0,1)$:
\begin{equation}\label{eq:ctnty_barcode}
  \liminf_{t\to t_{0}-} L(t) = \liminf_{t\to t_{0}+}L(t).
\end{equation}
While the barcode can be defined for any family $K_{t}$, we prefer to perturb $K_{t}$ slightly to $\bar{K}_{t}$ so that the following property holds: there is finite set $0<t_{1}<\dots<t_{k}<1$ so that for $t\ne t_{j}$, $\bar{K}_{t}$ is a \emph{strong Morse function}. Here ``strong'' means that each critical value is attained at a single critical point.

As explained in \cite{barannikov, chekanov_pushkar, fuchs_rutherford, pushkar_tyomkin}, the critical values of a family of strong Morse function vary smoothly $c_{1}(t)<\dots<c_{M}(t)$, and, if $x_{1}(t),\dots,x_{M}(t)$ are the corresponding critical points, there is an upper triangular change of basis 
\begin{equation*}
  y_{j}=\sum_{i\le j} a_{ij}x_{i}\text{ so that }d y_{j}=y_{\pi(j)}\text{ or }d y_{\pi(j)}=y_{j},
\end{equation*}
for a unique fixed-point-free involution $\pi:\set{1,\dots,M}\to \set{1,\dots,M}$, and where $d$ is the Morse differential (with $\Z/2$ coefficients). This decomposition is known as the \emph{Barannikov decomposition}, and it is straightforward to show that $\set{c_{j}(t),c_{\pi(j)}(t)}$ form the endpoints of a bar in the barcode of $\bar{K}_{t}$.

Recall from Lemma \ref{lemma:main_3_prop} that $K_{0}$ has a compact Morse-Bott critical manifold of value $0$, and all the other critical values have absolute value at least the minimal positive action of a Reeb chord of $\Lambda_{0}$. Call this minimal positive action $A$, and pick $\epsilon$ much smaller than $A$.

The small perturbation $\bar{K}_{0}$ will have a collection of critical values $\epsilon$-close to zero, and the rest will be $\epsilon$-close to numbers at least $A$ in absolute value. Let us call the critical values of the first kind ``Type 1'' and the others ``Type 2.'' In the Barannikov decomposition of $\bar{K}_{0}$, there must be a pair consisting of a Type 1 and a Type 2 critical value, otherwise the Morse homology of the cobordism $\set{-2\epsilon\le K\le 2\epsilon}$ would vanish, contradicting the well-known fact that the homology is isomorphic to the homology of the Morse-Bott submanifold of critical value $0$ (potentially with a grading shift).

Thus $\bar{K}_{0}$ has a bar of length $A-2\epsilon$ in its barcode, say attained by $c_{i}(0)<c_{j}(0)$. For $t<t_{1}$ the interval $[c_{i}(t),c_{j}(t)]$ is in the barcode for $\bar{K}_{t}$.

Applying Lemma \ref{lemma:cerf_ctnty}, we assume the singular Legendrian diagram $\bar{D}$ for $\bar{K}_{t}$ satisfies $\bar{D}\cap \set{q=t}\subset \set{p\in (\min h_{t}-\epsilon,\max h_{t}+\epsilon)}$. Clearly $(z,q,p)=(c_{i}(t),t,c_{i}'(t))$ lies in the singular Legendrian diagram for $\bar{K}_{t}$, and similarly for $j$, thus:
\begin{equation*}
  \ud{}{t}(c_{j}(t)-c_{i}(t))> -(\max h_{t}-\min h_{t})-2\epsilon.
\end{equation*}
In particular, if $L(t)$ denotes the length of the longest bar, 
\begin{equation*}
  \liminf_{t\to t_{1}-}L(t)> A-2\epsilon-2\epsilon t_{1}-\int_{0}^{t_{1}}\max h_{t}-\min h_{t}\,\d t.
\end{equation*}
Applying \eqref{eq:ctnty_barcode}, we conclude that as $t\to t_{1}+$ there is a bar of a similar length (say up to $2^{-1}\epsilon$). Repeating the argument for each singular time $t_{p}$ (with $2^{-1}\epsilon$ replaced by $2^{-p}\epsilon$), it can be shown that:
\begin{equation*}
  L(1) > A-5\epsilon-\int_{0}^{1}\max h_{t}-\min h_{t}\, \d t.
\end{equation*}
Since $\epsilon$ was arbitrary, we conclude that \emph{if the oscillation energy of $\Lambda_{t}$ is less than $A$, then $L(1)>0$ and hence there is a Reeb chord from $\Lambda_{0}$ to $\Lambda_{1}$}. This completes the proof of Theorem \ref{theorem:generating_function}.

\subsection{Proof of Theorem \ref{theorem:immersion_disjoining}}
\label{sec:proof_immersion}
Observe that any loop $j:\R/\Z\to B(1+\epsilon)$ which bounds area $1$ can be lifted to a Legendrian in $Q(1+\epsilon)$. Indeed, if $j$ has area $1$, then there is a function $f:\R/\Z\to \R/\Z$ so that $f^{*}\d t=j^{*}\lambda$, and $(f,j)$ is the desired Legendrian lift.

Let $\Lambda_{0}$ be the Legendrian lift corresponding to an embedded loop $j$ with area $1$ inside of $B(1+\epsilon)$ (e.g., one could take a lift of $\bd B(1)$). The idea is to deform $j$ through the space of immersed loops bounding signed area $1$ according to Figure \ref{fig:1} and \ref{fig:2}. Basically, one does the local deformation in Figure \ref{fig:1} to add the loop of area $A_{3}$ in Figure \ref{fig:2}. For the second part of the process, one expands the loop of area $A_{3}$ while simultaneously contracting the piece of area $A_{1}$ in such a way which keeps the total area $1$. The net process brings the new loop slightly within the original loop, with as small oscillation energy as desired; see \S\ref{sec:interpretation} for more details. This completes the proof of Theorem \ref{theorem:immersion_disjoining}.

\subsubsection{Lagrangian interpretation of the oscillation energy}
\label{sec:interpretation}

As in the above argument, let $(f_{t},j_{t}):\R/\Z\to \R/\Z\times \C$ be a deformation through immersed Legendrians, i.e., $j_{t}^{*}\lambda=f_{t}^{*}\d\theta$ for all $t$. Suppose $f_{t}$ always has degree $1$. Since the area bounded by $j_{t}$ is constant, $j_{t}$ is an \emph{exact deformation} of Lagrangians, i.e., there exist $\R$-valued $F_{t}$ so that $\omega_{j_{t}(x)}(j_{t}'(x),\d j_{t,x}(-))=\d F_{t,x}(-)$.

One defines the \emph{Lagrangian} oscillation energy as $\int_{0}^{1}\text{max}_{x}F_{t}(x)-\min_{x} F_{t}(x)\d t$. It is a straightforward exercise in the calculus of differential forms to show that the Lagrangian oscillation energy agrees with the Legendrian oscillation energy of $(f_{t},j_{t})$.

For arbitrary exact deformations of Lagrangian curves (not necessarily loops), if we change $j_{t}$ by $\varphi \circ j_{t}$ where $\varphi:\C\to \C$ is a contracting map $\varphi^{*}\omega=c\omega$, $c<1$, then it is easy to show that the Lagrangian oscillation energy changes by a factor of $c$. This observation implies that we can do the local deformation in Figure \ref{fig:1} with as small energy as desired (i.e., do the process at a normal scale, and then compose with $\varphi$ with a very small $c$). The details are left to the reader.

\begin{figure}[H]
  \centering
  \begin{tikzpicture}
    \node [inner sep=0pt, outer sep =0pt] at (0,0) {\includegraphics[scale=0.17]{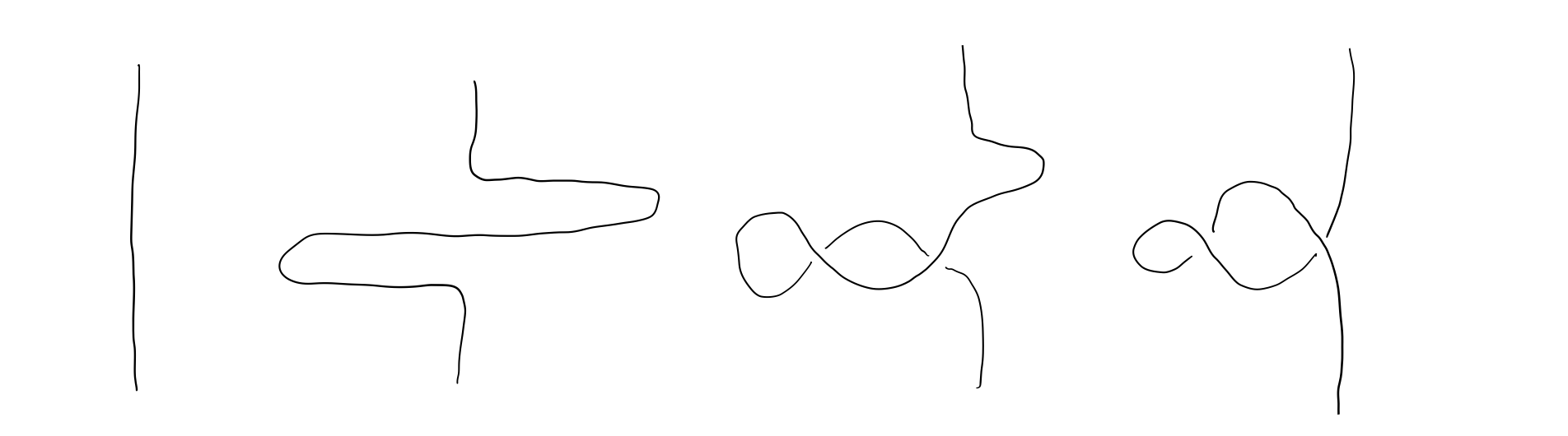}};
  \end{tikzpicture}
  \caption{A particular local model of an area preserving deformation of a curve in the plane. Note that the rightmost crossing \emph{changes its sign} in going from the third to the fourth stage.}
  \label{fig:1}
\end{figure}

\begin{figure}[H]
  \centering
  \begin{tikzpicture}
    \node at (0,0) [inner sep=0pt, outer sep =0pt] {\includegraphics[scale=0.17]{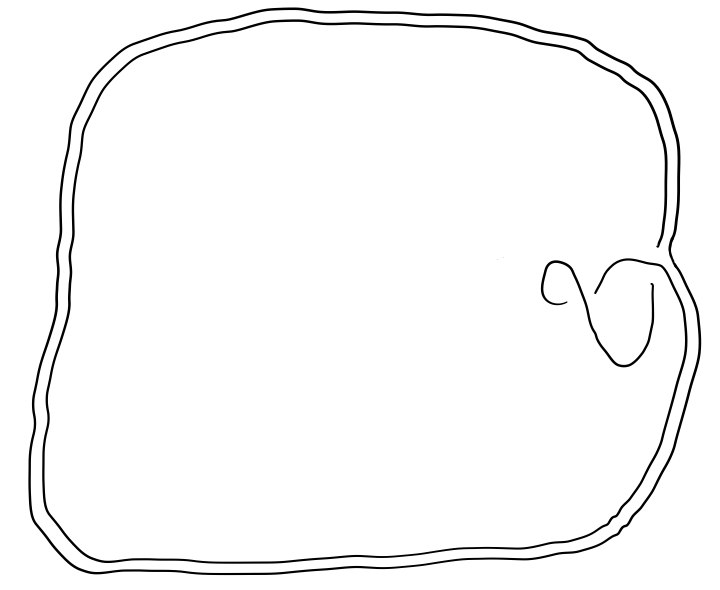}};
    \node (X0) at (2.5,0.8) [right] {$\phantom{A_{0}}$};
    \node (X1) at (X0.south west) [below right] {$A_{1}$};
    \node (X2) at (X1.south west) [below right] {$\phantom{A_{2}}$};
    \node (X3) at (X2.south west) [below right] {$A_{3}$};
    % \draw [-Latex](X0) -- (1.7,1.2);
    \draw [-Latex](X1) -- (1.2,0.6);
    \draw [-Latex](-2.3,-0.3) node[left]{$A_{2}$}-- (1.2,0.05);
    \draw [-Latex](X3) -- (1.5,-0.2);
  \end{tikzpicture}
  \caption{Isotoping a loop through immersed loops with area $1$. The process is area preserving provided that $A_{1}+2A_{3}-A_{2}$ equals the area of the original loop. By enlarging $A_{3}$ and shrinking $A_{1}$ we can bring the loop strictly inside the region bounded by the original loop.}
  \label{fig:2}
\end{figure}
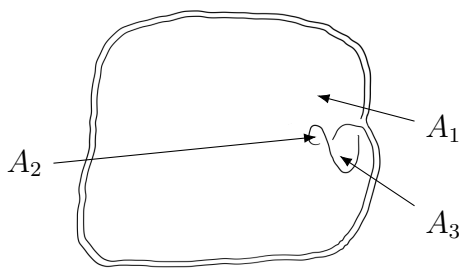
\subsection{Proof of Proposition \ref{prop:rizell_sullivan}}
\label{sec:proof_conjecture}
This section is inspired by \cite{rizell_sullivan_non_squeezing,rizell_sullivan_private_communication}.

Observe that $\R/\Z\times B(1)$ with form $\d t-\lambda$ is contactomorphic\footnote{Note that the diffeomorphism $(t,z)\mapsto (t,\bar{z})$ is a contactomorphism between the two forms $\d t\pm \lambda$.} to the standard $S^{3}$ minus a Hopf circle. Indeed, $S^{3}$ is the prequantization of $S^{2}$ with an area form of total area $1$, and so a symplectomorphism $B(1)\to S^{2}\setminus \set{\mathrm{pt}}$ extends to a strict contactomorphism between the total spaces of the prequantization bundles.

A Legendrian lift $\Lambda_{2}$ of a double cover $\R/\Z\to \bd B(1/2)$ is a standard unknot ($\mathtt{tb}(\Lambda_{2})=-1$) in $S^{3}$. This can be arranged by picking the embedding $B(1)\to S^{2}$ so $\bd B(1/2)$ is sent to the equator, recalling that the standard unknot in $S^{3}$ projects to a double cover of an equator in $S^{2}$.

Suppose that $\Lambda_{1}$ is Legendrian isotopic to a knot $\Lambda$ which lies in $\R/\Z\times B(1)\subset S^{3}$. The difference $\mathtt{tb}(\Lambda)-\mathtt{tb}(\Lambda_{2})$ can be computed as in \cite[\S3]{rizell_sullivan_non_squeezing}: if $\Sigma$ is a Seifert surface for $\Lambda-\Lambda_{2}$, then the homological intersection number of the Reeb push-off $\rho_{t}(\Lambda+\Lambda_{2})$ with $\Sigma$, as $t\to 0+$, equals $\mathtt{tb}(\Lambda)-\mathtt{tb}(\Lambda_{2})$. Indeed, we only require $\Sigma$ to be a bordism class with boundary $\Lambda-\Lambda_{2}$. 

The advantage of this approach is that $\Lambda-\Lambda_{2}$ is null-bordant in $\R/\Z\times \C$ (rather than just in $S^{3}$). It follows from \cite[Lemma 3.1]{rizell_sullivan_non_squeezing} that:
\begin{equation}\label{eq:diff}
  \mathtt{tb}(\Lambda)-\mathtt{tb}(\Lambda_{2})=\text{intersection number of $\rho_{t}(\Lambda_{1}+\Lambda_{2})$ with $\Sigma$, $\bd\Sigma=\Lambda_{1}-\Lambda_{2}$.}
\end{equation}
We claim that this difference equals $+1$. Then $\mathtt{tb}(\Lambda_{2})=-1$ implies $\mathtt{tb}(\Lambda)=0$, contradicting Bennequin's inequality \cite{bennequin} stating that $\mathtt{tb}(\Lambda)<0$ for all Legendrian unknots in $S^{3}$.

To compute the right hand side of \eqref{eq:diff}, let $r_{1}=1/\sqrt{\pi}$, $r_{2}=1/(\sqrt{2\pi})$, and consider the map:
\begin{equation*}
  f(s,t)=(t,(1-s)r_{2}e^{4\pi it}+sr_{1}e^{2\pi it})\in \R/\Z\times \C.
\end{equation*}
It is straightforward to check that $f$ is an embedding and hence defines a Seifert surface with $\bd f=\Lambda_{1}-\Lambda_{2}$. Moreover, the \emph{radial push-off} $\Lambda_{i}^{\epsilon}=e^{\epsilon}\Lambda_{i}$ induces the same framing as the Reeb push-off, and hence the intersection number of $\Lambda_{1}^{\epsilon}+\Lambda_{2}^{\epsilon}$ with $f$ computes \eqref{eq:diff}.

Clearly $f(s,t)\not\in \Lambda_{1}^{\epsilon}$ for all $s,t$. On the other hand:
\begin{equation*}
  f(s,t)\in \Lambda_{2}^{\epsilon}\iff t=0\text{ and }(1-s)r_{2}+sr_{1}=e^{\epsilon}r_{2}.
\end{equation*}
With a bit of thought, one sees that $T\Lambda_{2}^{\epsilon},\bd_{s}f,\bd_{t}f$ forms a positively oriented frame (using the contact orientation), and hence the right hand side of \eqref{eq:diff} equals $+1$. As explained above, this contradicts Bennequin's inequality and concludes the proof.

\subsection{Proof of Theorem \ref{theorem:imm_dichotomy}}
\label{sec:imm_dichotomy_proof}

As mentioned in \S\ref{sec:dichotomy}, the proof reduces to establishing Lemma \ref{lemma:chekanov_leg}. The reader is encouraged to refer to \cite[pp.\ 609]{chekanov_2000}, as our argument is essentially a direct copy of his. The argument goes as follows: suppose that $\delta$ is a pseudo-metric on $\mathscr{E}_{k}$ and $\delta(\Lambda_{0},\Lambda_{\infty})= 0$ for $\Lambda_{0}\ne \Lambda_{\infty}$. The goal is to show that $\delta$ is identically zero.

Pick a $1$-jet tubular neighborhood around $\Lambda_{0}$. This neighborhood of $\Lambda_{0}$ is covered by a union of cubes $Q(\epsilon)$ defined by $\abs{q}<\epsilon,\abs{p}<\epsilon,\abs{z}<\epsilon^{2}$, using local canonical coordinates $q,p,z$.

Shrinking $\epsilon$ if necessary, $\Lambda_{\infty}$ is disjoint from one of these cubes $Q(\epsilon)$. If $\Lambda_{1}$ is the $1$-jet of a function $z=f(q)$, $p=\d f(q)$ where $f$ is $C^{1}$ small and supported in $\abs{q}<\epsilon$, then \emph{we can find an ambient contactomorphism $\phi$ so that $\phi(\Lambda_{0})=\Lambda_{1}$ and $\phi(\Lambda_{\infty})=\Lambda_{\infty}$}. The desired contactomorphism is an extension of the Legendrian isotopy obtained by taking $1$-jets of $tf$ as $t$ ranges from $[0,1]$, which we can take to be supported in $Q(\epsilon)$.

Using $\delta(\Lambda_{1},\Lambda_{\infty})\le C(\phi)\delta(\Lambda_{0},\Lambda_{\infty})=0$, the triangle inequality yields $\delta(\Lambda_{0},\Lambda_{1})=0$.

Fix this special cube $Q(\epsilon)$. By applying canonical transformations (diffeomorphisms of the base lifted to contactomorphisms of $1$-jet space), we can cover a neighborhood of $\Lambda_{0}$ by finitely many contactomorphic images of $Q(\epsilon)$. If $\Lambda_{1}$ is the $1$-jet of a $C^{1}$ small function supported in \emph{any} of these cubes, then we conclude that $\delta(\Lambda_{0},\Lambda_{1})=0$.

Let $\Lambda_{f}$ be denote the $1$-jet of $f$. Given two $C^{1}$ small functions $f_{1},f_{2}$, there is a contactomorphism $\phi$ taking $\Lambda_{0},\Lambda_{f_{1}}$ to $\Lambda_{f_{2}},\Lambda_{f_{2}+f_{1}}$. Therefore, if $f_{1},f_{2}$ are supported in the above cubes, we have:
\begin{equation*}
  \delta(\Lambda_{0},\Lambda_{f_{1}+f_{2}})\le \delta(\Lambda_{0},\Lambda_{f_{2}})+\delta(\Lambda_{f_{2}},\Lambda_{f_{1}+f_{2}})\le \delta(\Lambda_{0},\Lambda_{f_{2}})+C(\phi)\delta(\Lambda_{0},\Lambda_{f_{1}})=0.
\end{equation*}
By iterating this argument, we conclude that \emph{any $1$-jet of a $C^{1}$ small function} has zero $\delta$-distance to $\Lambda_{0}$. Since $\mathscr{E}_{k}$ is path connected (in the $C^{1}$ topology), we conclude that $\delta=0$ on $\set{\Lambda_{0}}\times \mathscr{E}_{k}$, which implies the desired result.

\bibliography{citations}
\bibliographystyle{alpha}

\end{document}